\documentclass[12pt]{amsart}
\usepackage{graphicx}
\usepackage{amsmath}
\usepackage{amssymb}
\usepackage{setspace}
\usepackage{amsthm}
\newtheorem{thm}{Theorem}[section]
\newtheorem{cor}[thm]{Corollary}

\newtheorem{lem}[thm]{Lemma}
\newtheorem{prop}[thm]{Proposition}
\theoremstyle{definition}

\theoremstyle{remark}
\newtheorem{rem}[thm]{Remark}
\theoremstyle{conclusion}

\theoremstyle{conjecture}

\numberwithin{equation}{section}

\begin{document}
\title[Static Schr\"{o}dinger-Hartree-Maxwell type equations]{Classification of nonnegative solutions to static Schr\"{o}dinger-Hartree-Maxwell type equations}

\author{Wei Dai, Zhao Liu, Guolin Qin}

\address{School of Mathematics and Systems Science, Beihang University (BUAA), Beijing 100083, P. R. China, and LAGA, Universit¨¦ Paris 13 (UMR 7539), Paris, France}
\email{weidai@buaa.edu.cn}

\address{School of Mathematics and Computer Science, Jiangxi Science and Technology Normal University, Nanchang 330038, P. R. China, and Department of Mathematics, Yeshiva University, New York, NY, USA}
\email{liuzhao@mail.bnu.edu.cn}

\address{Institute of Applied Mathematics, Chinese Academy of Sciences, Beijing 100190, P. R. China, and University of Chinese Academy of Sciences, Beijing 100049, P. R. China}
\email{qinguolin18@mails.ucas.ac.cn}

\thanks{W. Dai is supported by the NNSF of China (No. 11971049), the Fundamental Research Funds for the Central Universities and the State Scholarship Fund of China (No. 201806025011); Z. Liu is supported by the NNSF of China (No. 11801237) and the State Scholarship Fund of China (No. 201808360005).}

\begin{abstract}
In this paper, we are mainly concerned with the physically interesting static Schr\"{o}dinger-Hartree-Maxwell type equations
\begin{equation*}
  (-\Delta)^{s}u(x)=\left(\frac{1}{|x|^{\sigma}}\ast |u|^{p}\right)u^{q}(x) \,\,\,\,\,\,\,\,\,\,\,\, \text{in} \,\,\, \mathbb{R}^{n}
\end{equation*}
involving higher-order or higher-order fractional Laplacians, where $n\geq1$, $0<s:=m+\frac{\alpha}{2}<\frac{n}{2}$, $m\geq0$ is an integer, $0<\alpha\leq2$, $0<\sigma<n$, $0<p\leq\frac{2n-\sigma}{n-2s}$ and $0<q\leq\frac{n+2s-\sigma}{n-2s}$. We first prove the super poly-harmonic properties of nonnegative classical solutions to the above PDEs, then show the equivalence between the PDEs and the following integral equations
\begin{equation*}
u(x)=\int_{\mathbb{R}^n}\frac{R_{2s,n}}{|x-y|^{n-2s}}\left(\int_{\mathbb{R}^{n}}\frac{1}{|y-z|^{\sigma}}u^p(z)dz\right)u^{q}(y)dy.
\end{equation*}
Finally, we classify all nonnegative solutions to the integral equations via the method of moving spheres in integral form. As a consequence, we obtain the classification results of nonnegative classical solutions for the PDEs. Our results completely improved the classification results in \cite{CD,DFQ,DL,DQ,Liu}. In critical and super-critical order cases (i.e., $\frac{n}{2}\leq s:=m+\frac{\alpha}{2}<+\infty$), we also derive Liouville type theorem.
\end{abstract}
\maketitle {\small {\bf Keywords:} Higher-order fractional Laplacians; Schr\"{o}dinger-Hartree-Maxwell equations; Classification of nonnegative solutions; Super poly-harmonic properties; Nonlocal nonlinearities; The method of moving spheres. \\

}

\section{Introduction}
\subsection{Classification results for nonnegative solutions in sub-critical order cases: $0<s<\frac{n}{2}$}
In this paper, we mainly consider nonnegative classical solutions to the following physically interesting static Schr\"{o}dinger-Hartree-Maxwell type equations involving higher-order or higher-order fractional Laplacians
\begin{equation}\label{PDE}
(-\Delta)^{s}u(x)=\left(\frac{1}{|x|^{\sigma}}\ast|u|^{p}\right)u^{q}(x) \,\,\,\,\,\,\,\,\,\,\,\, \text{in} \,\,\, \mathbb{R}^{n},
\end{equation}
where $n\geq1$, $0<s:=m+\frac{\alpha}{2}<\frac{n}{2}$, $m\geq0$ is an integer, $0<\alpha\leq2$, $0<\sigma<n$, $0<p\leq\frac{2n-\sigma}{n-2s}$ and $0<q\leq\frac{n+2s-\sigma}{n-2s}$. The higher-order or higher-order fractional Laplacians $(-\Delta)^{s}:={(-\Delta)}^m(-\Delta)^{\frac{\alpha}{2}}$. When $\sigma=4s$, $p=2$, $0<q\leq1$, \eqref{PDE} is called static Schr\"{o}dinger-Hartree type equations. When $\sigma=n-2s$, $p=\frac{n+2s}{n-2s}$, $0<q\leq\frac{4s}{n-2s}$, \eqref{PDE} is known as static Schr\"{o}dinger-Maxwell type equations. We say that equation \eqref{PDE} is in critical order if $s=\frac{n}{2}$, is in sub-critical order if $0<s<\frac{n}{2}$ and is in super-critical order if $\frac{n}{2}<s<+\infty$.

When $0<\alpha<2$, the nonlocal fractional Laplacians $(-\Delta)^{\frac{\alpha}{2}}$ is defined by
\begin{equation}\label{nonlocal defn}
  (-\Delta)^{\frac{\alpha}{2}}u(x)=C_{\alpha,n} \, P.V.\int_{\mathbb{R}^n}\frac{u(x)-u(y)}{|x-y|^{n+\alpha}}dy:=C_{\alpha,n}\lim_{\epsilon\rightarrow0}\int_{|y-x|\geq\epsilon}\frac{u(x)-u(y)}{|x-y|^{n+\alpha}}dy
\end{equation}
for any functions $u\in C^{[\alpha],\{\alpha\}+\epsilon}_{loc}\cap\mathcal{L}_{\alpha}(\mathbb{R}^{n})$ with arbitrarily small $\epsilon>0$, where $[\alpha]$ denotes the integer part of $\alpha$, $\{\alpha\}:=\alpha-[\alpha]$, the constant $C_{\alpha,n}=\left(\int_{\mathbb{R}^{n}}\frac{1-\cos(2\pi\zeta_{1})}{|\zeta|^{n+\alpha}}d\zeta\right)^{-1}$ and the function spaces
\begin{equation}\label{0-1}
  \mathcal{L}_{\alpha}(\mathbb{R}^{n}):=\left\{u: \mathbb{R}^{n}\rightarrow\mathbb{R}\,\Big|\,\int_{\mathbb{R}^{n}}\frac{|u(x)|}{1+|x|^{n+\alpha}}dx<\infty\right\}.
\end{equation}
The fractional Laplacians $(-\Delta)^{\frac{\alpha}{2}}$ can also be defined equivalently (see \cite{CLM}) by Caffarelli and Silvestre's extension method (see \cite{CS}) for $u\in C^{[\alpha],\{\alpha\}+\epsilon}_{loc}(\mathbb{R}^{n})\cap\mathcal{L}_{\alpha}(\mathbb{R}^{n})$.

We say $u$ is a classical solution to equation \eqref{PDE}, provided that $u\in C^{2m+\alpha}(\mathbb{R}^{n})$ if $\alpha=2$, $u\in C^{2m+[\alpha],\{\alpha\}+\epsilon}_{loc}(\mathbb{R}^{n})$ (with arbitrarily small $\epsilon>0$) if $0<\alpha<2$, and $u$ satisfies equation \eqref{PDE} pointwise in $\mathbb{R}^{n}$. Throughout this paper, we define $(-\Delta)^{m+\frac{\alpha}{2}}u:=(-\Delta)^{m}(-\Delta)^{\frac{\alpha}{2}}u$ for $u\in C^{2m+[\alpha],\{\alpha\}+\epsilon}_{loc}(\mathbb{R}^{n})\cap\mathcal{L}_{\alpha}(\mathbb{R}^{n})$ in the cases $0<\alpha<2$, where $(-\Delta)^{\frac{\alpha}{2}}u$ is defined by definition \eqref{nonlocal defn}. Due to the nonlocal feature of $(-\Delta)^{\frac{\alpha}{2}}$, we need to assume  $u\in C^{2m+[\alpha],\{\alpha\}+\epsilon}_{loc}(\mathbb{R}^{n})$ with arbitrarily small $\epsilon>0$ (merely $u\in C^{2m+[\alpha],\{\alpha\}}$ is not enough) to guarantee that $(-\Delta)^{\frac{\alpha}{2}}u\in C^{2m}(\mathbb{R}^{n})$ (see \cite{CLM,S}), and hence $u$ is a classical solution to equation \eqref{PDE} in the sense that $(-\Delta)^{m+\frac{\alpha}{2}}u$ is pointwise well-defined and continuous in the whole $\mathbb{R}^{n}$.

One should observe that both the fractional Laplacians $(-\Delta)^{\frac{\alpha}{2}}$ and the Hartree type nonlinearity are nonlocal in our equation \eqref{PDE}, which is quite different from most of the known results in previous literature. The fractional Laplacian is a nonlocal integral operator. It can be used to model diverse physical phenomena, such as anomalous diffusion and quasi-geostrophic flows, turbulence and water waves, molecular dynamics, and relativistic quantum mechanics of stars (see \cite{Co,CV} and the references therein). It also has various applications in probability and finance (see \cite{Be,CT} and the references therein). In particular, the fractional Laplacian can also be understood as the infinitesimal generator of a stable L\'{e}vy process (see \cite{Be}).

PDEs of the type \eqref{PDE} arise in the Hartree-Fock theory of the nonlinear Schr\"{o}dinger equations (see \cite{LS}). The solution $u$ to problem \eqref{PDE} is also a ground state or a stationary solution to the following dynamic Schr\"{o}dinger-Hartree equation
\begin{equation}\label{Hartree}
i\partial_{t}u+(-\Delta)^{m+\frac{\alpha}{2}}u=\left(\frac{1}{|x|^{\sigma}}\ast|u|^{p}\right)u^{q}, \qquad (t,x)\in\mathbb{R}\times\mathbb{R}^{n}
\end{equation}
involving higher-order or higher-order fractional Laplacians. The Schr\"{o}dinger-Hartree equations have many interesting applications in the quantum theory of large systems of non-relativistic bosonic atoms and molecules (see, e.g. \cite{FL}). The Schr\"{o}dinger equations and Hartree equations with Laplacians, poly-Laplacians or fractional Laplacians have been quite intensively studied, please refer to \cite{Karpman,LMZ,MXZ,MXZ3} and the references therein, in which the ground state solution can be regarded as a crucial criterion or threshold for global well-posedness and scattering in the focusing case. Therefore, the classification of solutions to \eqref{PDE} plays an important and fundamental role in the study of the focusing dynamic Schr\"{o}dinger-Hartree equations \eqref{Hartree} involving higher-order or higher-order fractional Laplacians.

The qualitative properties of solutions to fractional order or higher order elliptic equations have been extensively studied, for instance, see \cite{CDQ0,CDQ,CF,CFL,CGS,CL2,CLO,CLM,CY,DQ,GNN1,Lin,LZ,WX,Xu} and the references therein. There are also lots of literatures on the qualitative properties of solutions to Hartree or Choquard type equations of fractional or higher order, please see e.g. \cite{CD,CDZ,CL3,CW,DFHQW,DFQ,DL,DQ,Lei,Lieb,Liu,MS,MS1,MZ,XL} and the references therein. Liu proved in \cite{Liu} the classification results for positive classical solutions to \eqref{PDE} with $m=0$, $\alpha=2$, $\sigma=4\in(0,n)$, $p=2$ and $q=1$, by using the idea of considering the equivalent systems of integral equations instead, which was initially used in dealing with second order Choquard equations by Ma and Zhao \cite{MZ}. In \cite{CD}, under some weak decay assumptions, Cao and Dai considered the differential equations directly and classified nonnegative $C^{4}$ solutions to bi-harmonic equations \eqref{PDE} with $m=1$, $\alpha=2$, $\sigma=8\in(0,n)$, $p=2$ and $0<q\leq1$. In \cite{DQ}, under some weak integrability assumptions, Dai and Qin classified nonnegative classical solutions to third order equations \eqref{PDE} with $m=1$, $\alpha=1$, $\sigma=6\in(0,n)$, $p=2$ and $0<q\leq1$. Dai, Fang and Qin \cite{DFQ} classified all the positive classical solutions to \eqref{PDE} when $m=0$, $0<\alpha<2$, $\sigma=2\alpha\in(0,n)$, $p=2$ and $q=1$ via a direct method of moving planes for fractional Laplacians. In \cite{DL}, Dai and Liu established classification results for nonnegative classical solutions to Schr\"{o}dinger-Hartree-Maxwell type equations \eqref{PDE} in the cases that $m=0$, $0<\alpha\leq 2$, $\sigma=2\alpha\in(0,n)$, $p=2$ and $0<q\leq1$ or $m=0$, $0<\alpha\leq 2$, $\sigma=n-\alpha\in(0,n)$, $p=\frac{n+\alpha}{n-\alpha}$ and $0<q\leq\frac{2\alpha}{n-\alpha}$.

In this paper, we will completely improve the classification results in Liu \cite{Liu}, Cao and Dai \cite{CD}, Dai and Qin \cite{DQ}, Dai, Fang and Qin \cite{DFQ} and Dai and Liu \cite{DL}, and classified nonnegative classical solutions to Schr\"{o}dinger-Hartree-Maxwell type equations \eqref{PDE} in the full range $s:=m+\frac{\alpha}{2}\in(0,\frac{n}{2})$, $m\geq0$ is an integer, $0<\alpha\leq2$, $0<\sigma<n$, $0<p\leq\frac{2n-\sigma}{n-2s}$ and $0<q\leq\frac{n+2s-\sigma}{n-2s}$.

Equations \eqref{PDE} are closely related to the following integral equations
\begin{equation}\label{IE}
u(x)=\int_{\mathbb{R}^n}\frac{R_{2s,n}}{|x-y|^{n-2s}}\left(\int_{\mathbb{R}^{n}}\frac{1}{|y-z|^{\sigma}}u^p(z)dz\right)u^{q}(y)dy,
\end{equation}
where the Riesz potential's constants $R_{\gamma,n}:=\frac{\Gamma(\frac{n-\gamma}{2})}{\pi^{\frac{n}{2}}2^{\gamma}\Gamma(\frac{\gamma}{2})}$ for $0<\gamma<n$ (see \cite{Stein}). In fact, we will show that PDEs \eqref{PDE} and IEs \eqref{IE} are equivalent. In order to prove this equivalence, we have to derive the super poly-harmonic properties first.

It is well known that the super poly-harmonic properties of nonnegative solutions play a crucial role in establishing the integral representation formulae, Liouville type theorems and classification of solutions to higher order PDEs in $\mathbb{R}^{n}$ or $\mathbb{R}^{n}_{+}$ (see \cite{BGM,CD,CDQ0,CDQ,CF,CFL,CL4,CY,DQ,Lin,WX} and the references therein). For integer higher-order equations (i.e., $\alpha=2$ in \eqref{PDE}), the super poly-harmonic properties for nonnegative solutions usually can be derived via the ``spherical average, re-centers and iteration" arguments in conjunction with careful ODE analysis (we refer to \cite{CFL,Lin,WX}, see also \cite{CDQ,CF,CL4,DQ} and the references therein). However, for the fractional higher-order equations \eqref{PDE} (i.e., $0<\alpha<2$ in \eqref{PDE}), the super poly-harmonic properties are difficult to be derived. The reason for this is that $(-\Delta)^{\frac{\alpha}{2}}$ is nonlocal and $(-\Delta)^{\frac{\alpha}{2}}f(r)$ can not be calculated or expanded accurately ($0<\alpha<2$ and $f(r)$ is a radially symmetric function), thus the strategy for integer higher-order equations does not work any more for equations \eqref{PDE} involving higher-order fractional Laplacians. In \cite{CDQ0}, by taking full advantage of the Poisson representation formulae for $(-\Delta)^{\frac{\alpha}{2}}$ and developing some new integral estimates on the average $\int_{R}^{+\infty}\frac{R^{\alpha}}{r(r^{2}-R^{2})^{\frac{\alpha}{2}}}\bar{u}(r)dr$ and iteration techniques, Cao, Dai and Qin first introduced a new approach to overcome these difficulties and establish the super-harmonic properties of nonnegative solutions to PDEs involving higher-order fractional Laplacians.

Inspired by ideas from \cite{CDQ0}, we prove the super poly-harmonic properties for nonnegative classical solutions to \eqref{PDE} in both integer and fractional higher-order cases in a unified way (see Section 2).

\begin{thm}\label{Thm0}
Assume $n\geq2$, $0<s:=m+\frac{\alpha}{2}<+\infty$, $0<\alpha\leq2$, $m\geq1$ is an integer, $-\infty<\sigma<n$, $0<q<+\infty$, $0<p<+\infty$ if $0<\alpha<2$, $1\leq p<+\infty$ if $\alpha=2$. Suppose that $u$ is a nonnegative classical solution to \eqref{PDE}. Then, we have, for every $i=0,1,\cdots,m-1$,
\begin{equation}\label{0-2}
  (-\Delta)^{i+\frac{\alpha}{2}}u(x)\geq0, \qquad \forall \,\, x\in\mathbb{R}^{n}.
\end{equation}
\end{thm}

\begin{rem}\label{rem1}
In the fractional higher-order cases $0<\alpha<2$, Theorem \ref{Thm0} can be deduced from Theorem 1.1 in \cite{CDQ0} directly. In the integer higher-order cases $\alpha=2$, the results in Theorem \ref{Thm0} are new. We will give a unified proof for Theorem \ref{Thm0} in both integer and fractional higher-order cases in Section 2.
\end{rem}

From the super poly-harmonic properties of nonnegative solutions in Theorem \ref{Thm0}, by using the methods in \cite{CDQ,CFY,DFQ,ZCCY}, we can deduce the following equivalence between PDEs \eqref{PDE} and IEs \eqref{IE} (see Section 3).
\begin{thm}\label{equivalence}
Assume $n\geq2$, $0<s:=m+\frac{\alpha}{2}<\frac{n}{2}$, $0<\alpha\leq2$, $m\geq0$ is an integer, $-\infty<\sigma<n$, $0<q<+\infty$, $0<p<+\infty$ if $\alpha\in(0,2)$ and $m\geq1$ \emph{or} $\alpha\in(0,2]$ and $m=0$, $1\leq p<+\infty$ if $\alpha=2$ and $m\geq1$. Suppose that $u$ is a nonnegative classical solution to \eqref{PDE}, then $u$ is also a nonnegative solution to integral equation \eqref{IE}, and vice versa.
\end{thm}

By the equivalence between PDEs \eqref{PDE} and IEs \eqref{IE}, we can consider integral equations \eqref{IE} instead of PDEs \eqref{PDE}. By applying the methods of moving spheres in integral forms, we derive the following classification results for nonnegative solutions to IEs \eqref{IE} (See Section 4). For more literatures on the classification of solutions and Liouville type theorems for various PDE and IE problems via the methods of moving planes or spheres, please refer to \cite{CD,CDQ,CF,CFY,CGS,CL1,CL0,CL2,CLL,CLO,CLZ,DFHQW,DFQ,DL,DQ,GNN1,JLX,Li,Lin,LZ,MZ,Pa,Serrin,WX,Xu} and the references therein.
\begin{thm}\label{Thm}
Assume $n\geq1$, $0<s:=m+\frac{\alpha}{2}<\frac{n}{2}$, $0<\alpha\leq2$, $m\geq0$ is an integer, $0<\sigma<n$, $0<p\leq\frac{2n-\sigma}{n-2s}$ and $0<q\leq\frac{n+2s-\sigma}{n-2s}$. Suppose $u\in C(\mathbb{R}^{n})$ is a nonnegative solution to IE \eqref{IE}. In the critical case $p=\frac{2n-\sigma}{n-2s}$ and $q=\frac{n+2s-\sigma}{n-2s}$, we have either $u\equiv0$ or $u$ must have the following form
\begin{equation*}
	u(x)=\mu^{\frac{n-2s}{2}}Q\left(\mu(x-x_{0})\right) \qquad \text{for some} \,\,\, \mu>0 \,\,\, \text{and} \,\,\, x_{0}\in\mathbb{R}^{n},
\end{equation*}
where $Q(x)=\left(\frac{1}{R_{2s,n}I\left(\frac{\sigma}{2}\right)I\left(\frac{n-2s}{2}\right)}\right)^{\frac{n-2s}{2(n+2s-\sigma)}}\left(\frac{1}{1+|x|^{2}}\right)^{\frac{n-2s}{2}}$ with $I(\gamma):=\frac{\pi^{\frac{n}{2}}\Gamma\left(\frac{n-2\gamma}{2}\right)}{\Gamma(n-\gamma)}$ for $0<\gamma<\frac{n}{2}$. In the subcritical cases $0<p<\frac{2n-\sigma}{n-2s}$ or $0<q<\frac{n+2s-\sigma}{n-2s}$, the unique nonnegative solution to \eqref{IE} is $u\equiv0$ in $\mathbb{R}^{n}$.
\end{thm}

As a consequence of the equivalence in Theorem \ref{equivalence} and the classification results for IEs \eqref{IE} in Theorem \ref{Thm}, we obtain the following classification results for PDEs \eqref{PDE}.
\begin{cor}\label{Cor}
Assume $n\geq2$, $0<s:=m+\frac{\alpha}{2}<\frac{n}{2}$, $0<\alpha\leq2$, $m\geq0$ is an integer, $0<\sigma<n$, $0<q\leq\frac{n+2s-\sigma}{n-2s}$, $0<p\leq\frac{2n-\sigma}{n-2s}$ if $\alpha\in(0,2)$ and $m\geq1$ \emph{or} $\alpha\in(0,2]$ and $m=0$, $1\leq p\leq\frac{2n-\sigma}{n-2s}$ if $\alpha=2$ and $m\geq1$. Suppose $u$ is a nonnegative classical solution to PDE \eqref{PDE}. In the critical case $p=\frac{2n-\sigma}{n-2s}$ and $q=\frac{n+2s-\sigma}{n-2s}$, we have either $u\equiv0$ or $u$ must have the following form
\begin{equation*}
	u(x)=\mu^{\frac{n-2s}{2}}Q\left(\mu(x-x_{0})\right) \qquad \text{for some} \,\,\, \mu>0 \,\,\, \text{and} \,\,\, x_{0}\in\mathbb{R}^{n},
\end{equation*}
where $Q(x)=\left(\frac{1}{R_{2s,n}I\left(\frac{\sigma}{2}\right)I\left(\frac{n-2s}{2}\right)}\right)^{\frac{n-2s}{2(n+2s-\sigma)}}\left(\frac{1}{1+|x|^{2}}\right)^{\frac{n-2s}{2}}$ with $I(\gamma):=\frac{\pi^{\frac{n}{2}}\Gamma\left(\frac{n-2\gamma}{2}\right)}{\Gamma(n-\gamma)}$ for $0<\gamma<\frac{n}{2}$. In the subcritical cases $p<\frac{2n-\sigma}{n-2s}$ or $q<\frac{n+2s-\sigma}{n-2s}$, the unique nonnegative classical solution to \eqref{PDE} is $u\equiv0$ in $\mathbb{R}^{n}$.
\end{cor}
\begin{rem}
Our classification results in Corollary \ref{Cor} completely improve the classification results in Liu \cite{Liu}, Cao and Dai \cite{CD}, Dai and Qin \cite{DQ}, Dai, Fang and Qin \cite{DFQ} and Dai and Liu \cite{DL} to the full range of $n$, $s$, $\sigma$, $p$ and $q$.
\end{rem}

In the critical case $p=\frac{2n-\sigma}{n-2s}$ and $q=\frac{n+2s-\sigma}{n-2s}$, the classification of positive solutions to \eqref{PDE} would provide the best constants and extremal functions for the corresponding Hardy-Littlewood-Sobolev inequality (see \cite{Lieb}). We define the norm
\begin{equation}\label{0-3}
  \|u\|_{L^{V_{\sigma}}(\mathbb{R}^{n})}:=\left\|(V_{\sigma}\ast|u|^{\frac{2n-\sigma}{n-2s}})|u|^{\frac{2n-\sigma}{n-2s}}\right\|^{\frac{n-2s}{2(2n-\sigma)}}_{L^{1}(\mathbb{R}^{n})}
\end{equation}
for $0<\sigma<n$, $0<s<\frac{n}{2}$ with potential $V_{\sigma}=\frac{1}{|x|^{\sigma}}$. For any $u\in H^{s}(\mathbb{R}^{n})$, we have the following Hardy-Littlewood-Sobolev inequality (see \cite{Lieb,Stein})
\begin{equation}\label{HLS}
  \|u\|_{L^{V_{\sigma}}(\mathbb{R}^{n})}\leq S^{-1}_{\sigma,s,n}\|(-\Delta)^{\frac{s}{2}}u\|_{L^{2}(\mathbb{R}^{n})},
\end{equation}
where the best constant $S_{\sigma,s,n}$ is given by
\begin{equation}\label{S}
  S_{\sigma,s,n}=\inf_{u\in H^{s}(\mathbb{R}^{n}), \, u\neq0}\frac{\|(-\Delta)^{\frac{s}{2}}u\|_{L^{2}(\mathbb{R}^{n})}}{\|u\|_{L^{V_{\sigma}}(\mathbb{R}^{n})}}.
\end{equation}

When $p=\frac{2n-\sigma}{n-2s}$ and $q=\frac{n+2s-\sigma}{n-2s}$, the equation \eqref{PDE} is the corresponding Euler-Lagrange equation for the minimization problem described in \eqref{S}. By using the concentration-compactness arguments in \cite{L,L1} and the uniqueness of spherically symmetric positive solutions of the Euler-Lagrange equation \eqref{PDE} derived in Corollary \ref{Cor}, we obtain that
\begin{equation*}
  \mathcal{M}=\left\{e^{i\theta}\mu^{\frac{n-2s}{2}}Q\left(\mu(x-y)\right): \,\, \forall \, \theta\in(-\pi,\pi], \, \mu>0, \, y\in\mathbb{R}^{n}\right\}
\end{equation*}
is the set of minimizers for $S_{\sigma,s,n}$. Since minimization problem \eqref{S} can be attained by the extremal function $Q$, one can deduce from the definition of $S_{\sigma,s,n}$ and equation \eqref{PDE} that
\begin{equation}\label{indentity}
  \|(-\Delta)^{\frac{s}{2}}Q\|_{L^{2}(\mathbb{R}^{n})}=S_{\sigma,s,n}\|Q\|_{L^{V_{\sigma}}(\mathbb{R}^{n})}, \,\,\,\,\,\,\,\, \|(-\Delta)^{\frac{s}{2}}Q\|^{2}_{L^{2}(\mathbb{R}^{n})}=\|Q\|^{\frac{2(2n-\sigma)}{n-2s}}_{L^{V_{\sigma}}(\mathbb{R}^{n})},
\end{equation}
therefore, the best constant $S_{\sigma,s,n}$ for Hardy-Littlewood-Sobolev inequality \eqref{HLS} can be calculated explicitly as
\begin{equation}\label{bc}
  S_{\sigma,s,n}=\|(-\Delta)^{\frac{s}{2}}Q\|^{\frac{n+2s-\sigma}{2n-\sigma}}_{L^{2}(\mathbb{R}^{n})}=\|Q\|^{\frac{n+2s-\sigma}{n-2s}}_{L^{V_{\alpha}}(\mathbb{R}^{n})}.
\end{equation}

By further calculations and the accurate form of $Q$ , we can deduce from \eqref{bc} the following corollary.
\begin{cor}\label{Cor2}
The best constants in the Hardy-Littlewood-Sobolev inequality \eqref{HLS} is
\begin{equation}\label{Sharp}
   S_{\sigma,s,n}=\left[R_{2s,n}I\left(\frac{n-2s}{2}\right)\right]^{\frac{(n-2s)(n-\sigma)}{4s(2n-\sigma)}}
   \left[I\left(\frac{\sigma}{2}\right)\right]^{-\frac{n-2s}{2(2n-\sigma)}}\left(\widetilde{S}_{s,n}\right)^{-\frac{n(n+2s-\sigma)}{2s(2n-\sigma)}},
\end{equation}
where $\widetilde{S}_{s,n}:=\left(\frac{1}{2\sqrt{\pi}}\right)^{s}\left[\frac{\Gamma(n)}{\Gamma(\frac{n}{2})}\right]^{\frac{s}{n}}
\sqrt{\frac{\Gamma\left(\frac{n-2s}{2}\right)}{\Gamma\left(\frac{n+2s}{2}\right)}}$ is the best constant in the Sobolev inequality
\begin{equation}\label{Sobolev}
  \|f\|_{L^{\frac{2n}{n-2s}}(\mathbb{R}^{n})}\leq \widetilde{S}_{s,n}\|(-\Delta)^{\frac{s}{2}}f\|_{L^{2}(\mathbb{R}^{n})}
\end{equation}
for any $f\in H^{s}(\mathbb{R}^{n})$.
\end{cor}

\subsection{Liouville type theorem in critical and super-critical order cases: $\frac{n}{2}<s<+\infty$}
Now we consider the following critical and super-critical order static Schr\"{o}dinger-Hartree-Maxwell type equations involving higher-order or higher-order fractional Laplacians
\begin{equation}\label{PDE-sc}
(-\Delta)^{s}u(x)=\left(\frac{1}{|x|^{\sigma}}\ast|u|^{p}\right)u^{q}(x) \,\,\,\,\,\,\,\,\,\,\,\, \text{in} \,\,\, \mathbb{R}^{n},
\end{equation}
where $n\geq3$, $\frac{n}{2}\leq s:=m+\frac{\alpha}{2}<+\infty$, $m\geq1$ is an integer, $0<\alpha\leq2$, $-\infty<\sigma<n$, $0<p<+\infty$ and $0<q<+\infty$.

As an immediate consequence of the super poly-harmonic properties in Theorem \ref{Thm0}, by arguments developed by Chen, Dai and Qin \cite{CDQ} (see also \cite{CDQ0}), we can establish Liouville type theorem for nonnegative solutions to \eqref{PDE-sc} in both critical and super-critical order cases $\frac{n}{2}\leq s<+\infty$. For the particular case $\alpha=0$, Liouville type theorems for integer higher-order H\'{e}non-Hardy type equations in $\mathbb{R}^{n}$ or $\mathbb{R}^{n}_{+}$ have been derived by Chen, Dai and Qin \cite{CDQ} and Dai and Qin \cite{DQ3} in both critical and super-critical order cases. In \cite{CDQ0}, Cao, Dai and Qin extended the results in \cite{CDQ} to general fractional higher-order cases $0<\alpha<2$ and general nonlinearities $f(x,u,Du,\cdots)$.

For the critical and super-critical order cases, we have the following Liouville type theorem for nonnegative solutions to \eqref{PDE-sc}.
\begin{thm}\label{Thm3}
Assume $n\geq3$, $\frac{n}{2}\leq s:=m+\frac{\alpha}{2}<+\infty$, $0<\alpha\leq2$, $m\geq1$ is an integer, $-\infty<\sigma<n$, $0<q<+\infty$, $0<p<+\infty$ if $0<\alpha<2$, $1\leq p<+\infty$ if $\alpha=2$. Suppose that $u$ is a nonnegative classical solution to \eqref{PDE-sc}, then $u\equiv0$ in $\mathbb{R}^{n}$.
\end{thm}

\begin{rem}\label{rem2}
In the fractional higher-order cases $0<\alpha<2$, Theorem \ref{Thm3} can be deduced from Theorem 1.14 in \cite{CDQ0} directly. In the integer higher-order cases $\alpha=2$, the results in Theorem \ref{Thm3} are new. We will give a unified proof for Theorem \ref{Thm3} in both integer and fractional higher-order cases in Section 5.
\end{rem}

This paper is organized as follows. In Section 2, we will carry out our proof of Theorem \ref{Thm0}. In Section 3, we will prove Theorem \ref{equivalence}. Section 4 and 5 are devoted to proving Theorem \ref{Thm} and Theorem \ref{Thm3} respectively.

Throughout this paper, we will use $C$ to denote a general positive constant that may depend on $n$, $\alpha$, $m$, $\sigma$, $p$, $q$ and $u$, and whose value may differ from line to line.

\section{Super poly-harmonic properties}
In this section, we show super poly-harmonic properties for nonnegative classical solutions to equations \eqref{PDE}, i.e. Theorem \ref{Thm0}.

Assume $n\geq2$, $0<s:=m+\frac{\alpha}{2}<+\infty$, $0<\alpha\leq2$, $m\geq1$ is an integer, $-\infty<\sigma<n$, $0<q<+\infty$, $0<p<+\infty$ if $0<\alpha<2$, $1\leq p<+\infty$ if $\alpha=2$. Suppose that $u$ is a nonnegative classical solution to \eqref{PDE}. It follows from equations \eqref{PDE} that $\int_{\mathbb{R}^{n}}\frac{u^{p}(x)}{|x|^{\sigma}}dx<+\infty$. Let $u_1(x):=(-\Delta)^{\frac{\alpha}{2}}u(x)$ and $u_i(x):=(-\Delta)^{i-1}u_1(x)$ for $i=2,\cdots, m$. Then, from equations \eqref{PDE}, we have
\begin{equation}\label{2-1}
\left\{{\begin{array}{l} {(-\Delta)^{\frac{\alpha}{2}} u(x)=u_1(x) \quad \text{in} \,\, \mathbb{R}^{n},} \\  {} \\ {-\Delta u_{1}(x)=u_{2}(x) \quad \text{in} \,\, \mathbb{R}^{n},} \\ \cdots\cdots \\ {-\Delta u_{m}(x)=\Big(\frac{1}{|x|^{\sigma}}\ast|u|^{p}\Big)u^{q}(x)\geq0 \quad \text{in} \,\, \mathbb{R}^{n}.} \\ \end{array}}\right.
\end{equation}
Our aim is to prove that $u_i\geq 0$ in $\mathbb{R}^n$ for every $i=1,\cdots,m$. We will carry out the proof by discussing two different cases $\alpha=2$ and $0<\alpha<2$ separately.

\emph{Case(i).} We first consider the case that $\alpha=2$.

This case is more simpler, since we do not need to deal with the nonlocal fractional Laplacians. We first show that $u_m=(-\Delta)^{m}u\geq0$ by contradiction arguments. Indeed, suppose not, then there is a point $x_0\in\mathbb{R}^{n}$ such that $u_m(x_0)<0$. Without loss of generality, we may assume $x_0=0$ due to the translation invariance of equations \eqref{PDE}.

Now, let
\begin{equation}\label{2-2}
\bar{f}(r)=\bar{f}\big(|x|\big):=\frac{1}{|\partial B_{r}(0)|}\int_{\partial B_{r}(0)}f(x)d\sigma,
\end{equation}
be the spherical average of $f$ with respect to the center $0$. Then, by the well-known property $\overline{\Delta u}=\Delta\bar{u}$ and \eqref{2-1}, we have, for any $r\geq0$,
\begin{equation}\label{2-3}
\left\{{\begin{array}{l} {-\Delta\overline u(r)=\overline{u_1}(r)}, \\  {} \\ {-\Delta\overline{u_{1}}(r)=\overline{u_{ 2}}(r)}, \\ \cdots\cdots \\ {-\Delta\overline{u_{m}}(r)=\overline{\Big(\frac{1}{|x|^{\sigma}}\ast|u|^{p}\Big)u^{q}}(r)\geq0}. \\ \end{array}}\right.
\end{equation}

From last equation of \eqref{2-3}, one has
\begin{equation}\label{2-4}
-\frac{1}{r^{n-1}}\Big(r^{n-1}\overline{u_{ m}}\,'(r)\Big)'\geq0, \qquad \forall \, r\geq0.
\end{equation}
Integrating both sides of \eqref{2-4} twice gives
\begin{equation}\label{2-5}
 \overline{u_{m}}(r)\leq\overline{u_{m}}(0)=u_{m}(0)=:-c_{0}<0,
\end{equation}
for any $r\geq0$. Then from the last but one equation of \eqref{2-3} we derive
\begin{equation}\label{2-6}
-\frac{1}{r^{n-1}}\Big(r^{n-1}\overline{u_{m-1}}\,'(r)\Big)'\leq-c_0, \qquad \forall \, r\geq0.
\end{equation}
Again, by integrating both sides of \eqref{2-6} twice, we arrive at
\begin{equation}\label{2-7}
\overline{u_{m-1}}(r)\geq\overline{u_{m-1}}(0)+c_1 r^2, \qquad \forall \, r\geq0,
\end{equation}
where $c_1=\frac{c_0}{2n}>0$. Continuing this way, we finally obtain that
\begin{equation}\label{2-8}
{(-1)}^{(m+1)}\overline{u}(r)\geq a_{m}r^{2m}+a_{m-1}r^{2(m-1)}+\cdots+a_0, \qquad \forall \, r\geq0,
\end{equation}
where $a_m>0$. If $m$ is even, then \eqref{2-8} implies $\overline{u}(r)<0$ for $r$ sufficiently large, this contradicts with $u\geq0$ in $\mathbb{R}^{n}$. In the case that $m$ is odd, we can see from \eqref{2-8} that there exists $R_0>1$ sufficient large such that $\overline{u}(r)\geq Cr^{2m}$ for all $r>R_0$. Obviously, this will contradict with the integrability $\int_{\mathbb{R}^n}\frac{u^{p}(x)}{|x|^{\sigma}}<+\infty$. In fact, since $p\geq1$, one has
\begin{align}\label{2-0}
+\infty>\int_{\mathbb{R}^n}\frac{u^{p}(x)}{|x|^{\sigma}}dx&=C\int_{0}^{+\infty}\overline{u^p}(r)r^{n-\sigma-1}dr \\
&\geq C\int_{R_{0}}^{+\infty}{\overline{u}}^{p}(r)r^{n-\sigma-1}dr\nonumber \\
&\geq C\int_{R_{0}}^{+\infty}r^{2pm-1}dr=+\infty,\nonumber
\end{align}
which is absurd. Therefore, we must have $u_m=(-\Delta)^{m}u\geq0$. Next, we can prove that $u_{m-1}\geq0$ though entirely similar procedure as above. Continuing this way, we obtain that $u_i=(-\Delta)^{i}u\geq0$ for every $i=1,2,\cdots,m$.

\emph{Case(ii).} Next, we consider the cases that $0<\alpha<2$.

In such cases, we will first prove $u_m\geq0$ by contradiction arguments. If not, then there exists $x_0\in\mathbb{R}^n$ such that $u_m(x_0)<0$. We can still assume that $x_0=0$. As in the proof of Case (i) ($\alpha=2$), by taking spherical average w.r.t. center $0$ to all equations except the first equation in \eqref{2-1}, we have
\begin{equation}\label{2-9}
\left\{{\begin{array}{l} {{(-\Delta)}^{\frac{\alpha}{2}}u(x)=u_1(x) \quad \text{in} \,\, \mathbb{R}^{n},} \\  {} \\ {-\Delta\overline{u_{1}}(r)=\overline{u_{2}}(r), \quad \forall \, r\geq0,} \\ \cdots\cdots \\ {-\Delta\overline{u_{m}}(r)=\overline{\Big(\frac{1}{|x|^{\sigma}}\ast|u|^{p}\Big)u^{q}}(r)\geq0, \quad \forall \, r\geq0.} \\ \end{array}}\right.
\end{equation}
Through a similar argument as in deriving \eqref{2-8}, we deduce that
\begin{equation}\label{2-19}
(-1)^{m}\overline{u_1}(r)\geq a_{m-1}r^{2(m-1)}+\cdots+a_0, \qquad \forall \, r\geq0,
\end{equation}
where $a_{m-1}>0$. Hence, we have
\begin{equation}\label{2-10}
{(-1)}^{m}\overline{u_1}(r)\geq Cr^{2(m-1)},
\end{equation}
for any $r>R_0$ with $R_0$ sufficiently large. One should observe that it is very difficult to take spherical average to the first equation in \eqref{2-9}, since the fractional Laplacian $(-\Delta)^{\frac{\alpha}{2}}$ is a nonlocal operator. Inspired by a recent work by Cao, Dai and Qin \cite{CDQ0}, instead of taking spherical average, we will apply the Green-Poisson representation formulae for $(-\Delta)^{\frac{\alpha}{2}}$ to the first equation of \eqref{2-9} to overcome this difficulty. From the first equation in \eqref{2-9}, we conclude that, for arbitrary $R>0$,
\begin{equation}\label{2-20}
  u(x)=\int_{B_{R}(0)}G^\alpha_R(x,y)u_{1}(y)dy+\int_{|y|>R}P^{\alpha}_{R}(x,y)u(y)dy, \qquad \forall \,\, x\in B_{R}(0),
\end{equation}
where the Green's function for $(-\Delta)^{\frac{\alpha}{2}}$ with $0<\alpha<2$ on $B_R(0)$ is given by
\begin{equation}\label{2-21}
G^\alpha_R(x,y):=\frac{C_{n,\alpha}}{|x-y|^{n-\alpha}}\int_{0}^{\frac{t_{R}}{s_{R}}}\frac{b^{\frac{\alpha}{2}-1}}{(1+b)^{\frac{n}{2}}}db
\,\,\,\,\,\,\,\,\, \text{if} \,\, x,y\in B_{R}(0)
\end{equation}
with $s_{R}=\frac{|x-y|^{2}}{R^{2}}$, $t_{R}=\left(1-\frac{|x|^{2}}{R^{2}}\right)\left(1-\frac{|y|^{2}}{R^{2}}\right)$, and $G^{\alpha}_{R}(x,y)=0$ if $x$ or $y\in\mathbb{R}^{n}\setminus B_{R}(0)$ (see \cite{K}), and the Poisson kernel $P^{\alpha}_{R}(x,y)$ for $(-\Delta)^{\frac{\alpha}{2}}$ in $B_{R}(0)$ is defined by $P^{\alpha}_{R}(x,y):=0$ for $|y|<R$ and
\begin{equation}\label{2-22}
  P^{\alpha}_{R}(x,y):=\frac{\Gamma(\frac{n}{2})}{\pi^{\frac{n}{2}+1}}\sin\frac{\pi\alpha}{2}\left(\frac{R^{2}-|x|^{2}}{|y|^{2}-R^{2}}\right)^{\frac{\alpha}{2}}\frac{1}{|x-y|^{n}}
\end{equation}
for $|y|>R$ (see \cite{CLM}). Taking $x=0$ in \eqref{2-20} gives
\begin{align}\label{2-11}
u(0)&=\int_{B_R(0)}\frac{C_{n,\alpha}}{|y|^{n-\alpha}}\left(\int_0^{\frac{R^2}{|y|^2}-1}\frac{b^{\frac{\alpha}{2}-1}}{(1+b)^{\frac{n}{2}}}db\right)u_1(y)dy +C'_{n,\alpha}\int_{|y|>R}\frac{R^\alpha}{{(|y|^2-R^2)}^{\frac{\alpha}{2}}}\frac{u(y)}{|y|^n}dy\\
&=\int_{0}^{R}\frac{\widetilde{C}_{n,\alpha}}{r^{1-\alpha}}\left(\int_0^{\frac{R^2}{r^2}-1} \frac{b^{\frac{\alpha}{2}-1}}{(1+b)^{\frac{n}{2}}}db\right)\overline{u_1}(r)dr +\overline{C}_{n,\alpha}\int_{R}^{+\infty}\frac{R^\alpha\bar{u}(r)}{r{(r^2-R^2)}^{\frac{\alpha}{2}}}dr. \nonumber
\end{align}
One can easily observe that for $0<r\leq\frac{R}{2}$, $\frac{R^2}{r^2}-1\geq3$ and hence $\int_{0}^{\frac{R^2}{r^2}-1}\frac{b^{\frac{\alpha}{2}-1}}{(1+b)^{\frac{n}{2}}}db\geq\int_{0}^{3}\frac{b^{\frac{\alpha}{2}-1}}{(1+b)^{\frac{n}{2}}}db=:C_1>0$. For $\frac{R}{2}<r<R$, one has $0<\frac{R^2}{r^2}-1<3$, thus $\int_{0}^{\frac{R^2}{r^2}-1}\frac{b^{\frac{\alpha}{2}-1}}{(1+b)^{\frac{n}{2}}}db>\int_{0}^{\frac{R^2}{r^2}-1}\frac{b^{\frac{\alpha}{2}-1}}{2^n}db
=:C_2{\left(\frac{R^2-r^2}{r^2}\right)}^{\frac{\alpha}{2}}$. Thus, we conclude that
\begin{equation}\label{2-12}
\int_0^{\frac{R^2}{r^2}-1}\frac{b^{\frac{\alpha}{2}-1}}{(1+b)^{\frac{n}{2}}}db\geq C_1\chi_{0<r\leq\frac{R}{2}}+C_2\chi_{\frac{R}{2}<r<R}{\left(\frac{R^2-r^2}{r^2}\right)}^{\frac{\alpha}{2}}
\end{equation}
for any $0<r<R$. Then, from \eqref{2-10}, \eqref{2-11} and \eqref{2-12}, we derive that, for $R>2R_0$,
\begin{align}\label{2-13}
&\quad{(-1)}^{(m+1)}\int_{R}^{+\infty}  \frac{R^\alpha\bar{u}(r)}{r{(r^2-R^2)}^{\frac{\alpha}{2}}} dr \\
&=C\int_0^R \frac{1}{r^{1-\alpha}}\left(\int_0^{\frac{R^2}{r^2}-1} \frac{b^{\frac{\alpha}{2}-1}}{(1+b)^{\frac{n}{2}}}db\right) {(-1)}^{m}\overline{u_1}(r)dr+{(-1)}^{(m+1)}u(0)\nonumber\\
&\geq C\int_{R_0}^R \frac{1}{r^{1-\alpha}}\left[C_1 \chi_{0<r<\frac{R}{2}}
+C_2\chi_{\frac{R}{2}<r<R}{\left(\frac{R^2-r^2}{r^2}\right)}^{\frac{\alpha}{2}}\right] r^{2(m-1)}dr\nonumber \\
&\,\,\,\,\,\,\,+C\int_{0}^{R_0}\frac{1}{r^{1-\alpha}}\left(\int_{0}^{\frac{R^2}{r^2}-1}\frac{b^{\frac{\alpha}{2}-1}}{(1+b)^{\frac{n}{2}}}db\right)
(-1)^{m}\overline{u_1}(r)dr+{(-1)}^{(m+1)}u(0) \nonumber\\
&\geq CR^{2m-2+\alpha}-C_0+{(-1)}^{(m+1)}u(0)\nonumber.
\end{align}
It is easy to see that if $m$ is even, \eqref{2-13} implies that $\int_{R}^{+\infty}\frac{R^\alpha\bar{u}(r)}{r{(r^2-R^2)}^{\frac{\alpha}{2}}}dr<0$ for $R$ sufficiently large, which contradicts with $u\geq0$. Therefore, we only need to consider the case that $m$ is odd. In such cases, \eqref{2-13} implies
\begin{equation}\label{2-14}
\int_{R}^{+\infty}\frac{R^\alpha\bar{u}(r)}{r{(r^2-R^2)}^{\frac{\alpha}{2}}} dr\geq CR^{2m-2+\alpha}=CR^{2s-2}
\end{equation}
for $R$ sufficiently large. Since $u\in \mathcal{L}_{\alpha}(\mathbb{R}^{n})$, we have
\begin{equation} \label{2-15}
\int_{|x|>1}\frac{u(x)}{|x|^{n+\alpha}}dx=C\int_1^{+\infty} \frac{\bar{u}(r)}{r^{1+\alpha}}dr<+\infty.
\end{equation}
Then, by \eqref{2-15} and the fact that $2s-2\geq \alpha$, for $R$ sufficiently large, we have
\begin{align}\label{2-16}
\int_{R}^{2R}\frac{R^\alpha\bar{u}(r)}{r{(r^2-R^2)}^{\frac{\alpha}{2}}} dr
&\geq CR^{2s-2}-\int_{2R}^{+\infty} \frac{R^\alpha\bar{u}(r)}{r{(r^2-R^2)}^{\frac{\alpha}{2}}} dr\\
&\geq CR^{2s-2}-C'R^\alpha\int_{2R}^{+\infty} \frac{\bar{u}(r)}{r^{1+\alpha}}dr \nonumber\\
&\geq CR^{2s-2}-o_R(1)R^\alpha\nonumber\\
&\geq CR^{2s-2}.\nonumber
\end{align}
On the one hand, by \eqref{2-15}, we have
\begin{align}\label{2-17}
\int_1^{+\infty}\frac{1}{R^{1+\alpha}}\int_R^{2R}\frac{R^{\alpha}\bar{u}(r)}{r{(r^2-R^2)}^{\frac{\alpha}{2}}}drdR
&=\int_1^{+\infty}\frac{\bar{u}}{r}\int_{\frac{r}{2}}^{r}\frac{1}{R{(r^2-R^2)}^{\frac{\alpha}{2}}}dRdr\\
&\leq C\int_1^{+\infty}\frac{\bar{u}(r)}{r^{1+\alpha}}dr<+\infty.\nonumber
\end{align}
On the other hand, by \eqref{2-16}, we derive
\begin{equation}\label{2-18}
\int_1^{+\infty}\frac{1}{R^{1+\alpha}}\int_R^{2R}\frac{R^{\alpha}\bar{u}(r)}{r{(r^2-R^2)}^{\frac{\alpha}{2}}} drdR\geq\int_{N}^{+\infty}\frac{1}{R}dR=+\infty,
\end{equation}
where $N$ is sufficiently large such that \eqref{2-16} holds for any $R>N$. Combining \eqref{2-17} with \eqref{2-18}, we get a contradiction. Hence, we must have $u_m\geq 0$ in $\mathbb{R}^n$. One should observe that, in the proof of $u_{m}\geq0$, we have mainly used the property $-\Delta u_m\geq 0$. Therefore, through a similar argument as above, one can prove that $u_{m-1}\geq 0$. Continuing this way, we obtain that $u_i=(-\Delta)^{i-1+\frac{\alpha}{2}}u\geq0$ for every $i=1,2,\cdots,m$. This completes the proof of Theorem \ref{Thm0}.

\section{equivalence between PDE and IE}
In this section, we will prove the equivalence between PDEs \eqref{PDE} and IEs \eqref{IE}, i.e. Theorem \ref{equivalence}. We only need to show that any nonnegative classical solution $u$ to \eqref{PDE} also satisfies the integral equation \eqref{IE}. From the super poly-harmonic properties in Theorem \ref{Thm0}, we have already known that $u_i:=(-\Delta)^{i-1+\frac{\alpha}{2}}u\geq 0$ in $\mathbb{R}^{n}$ for every $i=1,\cdots,m$.

We will first show that $u_m=(-\Delta)^{m-1+\frac{\alpha}{2}}u$ satisfies the following integral equation
\begin{equation}\label{3-1}
u_m(x)=\int_{\mathbb{R}^{n}}\frac{R_{2,n}}{|x-y|^{n-2}}\left(\frac{1}{|\cdot|^{\sigma}}\ast|u|^{p}\right)(y)u^{q}(y)dy, \,\,\,\,\,\,\,\,\,\, \forall \,\, x\in\mathbb{R}^{n},
\end{equation}
where the Riesz potential's constants $R_{\gamma,n}:=\frac{\Gamma(\frac{n-\gamma}{2})}{\pi^{\frac{n}{2}}2^{\gamma}\Gamma(\frac{\gamma}{2})}$ for $0<\gamma<n$ (see \cite{Stein}).

To this end, for arbitrary $R>0$, let $f_{1}(u)(x):=\left(\frac{1}{|\cdot|^{\sigma}}\ast|u|^{p}\right)(x)u^{q}(x)$ and
\begin{equation}\label{2c2}
v_{1}^{R}(x):=\int_{B_R(0)}G^{2}_R(x,y)f_{1}(u)(y)dy,
\end{equation}
where the Green's function for $-\Delta$ on $B_R(0)$ is given by
\begin{equation}\label{Green}
  G^{2}_R(x,y)=R_{2,n}\left[\frac{1}{|x-y|^{n-2}}-\frac{1}{\left(|x|\cdot\Big|\frac{Rx}{|x|^{2}}-\frac{y}{R}\Big|\right)^{n-2}}\right], \,\,\,\,\,\, \text{if} \,\, x,y\in B_{R}(0),
\end{equation}
and $G^{2}_{R}(x,y)=0$ if $x$ or $y\in\mathbb{R}^{n}\setminus B_{R}(0)$. Then, we can derive that $v_{1}^{R}\in C^{2}(\mathbb{R}^{n})$ and satisfies
\begin{equation}\label{2c3}\\\begin{cases}
-\Delta v_{1}^{R}(x)=\left(\frac{1}{|\cdot|^{\sigma}}\ast|u|^{p}\right)(x)u^{q}(x),\ \ \ \ x\in B_R(0),\\
v_{1}^{R}(x)=0,\ \ \ \ \ \ \ x\in \mathbb{R}^{n}\setminus B_R(0).
\end{cases}\end{equation}
Let $w_{1}^R(x):=(-\Delta)^{m-1+\frac{\alpha}{2}}u(x)-v_{1}^R(x)$. By Theorem \ref{Thm0}, \eqref{PDE} and \eqref{2c3}, we have $w_{1}^R\in C^{2}(\mathbb{R}^{n})$ and satisfies
\begin{equation}\label{2c4}\\\begin{cases}
-\Delta w_{1}^R(x)=0,\ \ \ \ x\in B_R(0),\\
w_{1}^{R}(x)\geq0, \,\,\,\,\, x\in \mathbb{R}^{n}\setminus B_R(0).
\end{cases}\end{equation}
By maximum principle, we deduce that for any $R>0$,
\begin{equation}\label{2c5}
  w_{1}^R(x)=(-\Delta)^{m-1+\frac{\alpha}{2}}u(x)-v_{1}^{R}(x)\geq0, \,\,\,\,\,\,\, \forall \,\, x\in\mathbb{R}^{n}.
\end{equation}
Now, for each fixed $x\in\mathbb{R}^{n}$, letting $R\rightarrow\infty$ in \eqref{2c5}, we have
\begin{equation}\label{2c6}
(-\Delta)^{m-1+\frac{\alpha}{2}}u(x)\geq\int_{\mathbb{R}^{n}}\frac{R_{2,n}}{|x-y|^{n-2}}f_{1}(u)(y)dy=:v_{1}(x)\geq0.
\end{equation}
Take $x=0$ in \eqref{2c6}, we get
\begin{equation}\label{2c7}
  \int_{\mathbb{R}^{n}}\left(\frac{1}{|\cdot|^{\sigma}}\ast|u|^{p}\right)(y)\frac{u^{q}(y)}{|y|^{n-2}}dy<+\infty.
\end{equation}
One can easily observe that $v_{1}\in C^{2}(\mathbb{R}^{n})$ is a solution of
\begin{equation}\label{2c8}
-\Delta v_{1}(x)=\left(\frac{1}{|x|^{\sigma}}\ast|u|^{p}\right)u^{q}(x),  \,\,\,\,\,\,\, x\in \mathbb{R}^n.
\end{equation}
Define $w_{1}(x):=(-\Delta)^{m-1+\frac{\alpha}{2}}u(x)-v_{1}(x)$. Then, by \eqref{PDE}, \eqref{2c6} and \eqref{2c8}, we have $w_{1}\in C^{2}(\mathbb{R}^{n})$ and satisfies
\begin{equation}\label{2c9}\\\begin{cases}
-\Delta w_{1}(x)=0, \,\,\,\,\,  x\in \mathbb{R}^n,\\
w_{1}(x)\geq0, \,\,\,\,\,\,  x\in \mathbb{R}^n.
\end{cases}\end{equation}
From Liouville theorem for harmonic functions, we can deduce that
\begin{equation}\label{2c10}
   w_{1}(x)=(-\Delta)^{m-1+\frac{\alpha}{2}}u(x)-v_{1}(x)\equiv C_{1}\geq0.
\end{equation}
Therefore, we have
\begin{eqnarray}\label{2c11}
  (-\Delta)^{m-1+\frac{\alpha}{2}}u(x)&=&\int_{\mathbb{R}^{n}}\frac{R_{2,n}}{|x-y|^{n-2}}\left(\frac{1}{|\cdot|^{\sigma}}\ast|u|^{p}\right)(y)u^{q}(y)dy+C_{1} \\
  &=:&f_{2}(u)(x)\geq C_{1}\geq0. \nonumber
\end{eqnarray}

Next, for arbitrary $R>0$, let
\begin{equation}\label{2c12}
v_{2}^R(x):=\int_{B_R(0)}G^{2}_R(x,y)f_{2}(u)(y)dy.
\end{equation}
Then, we can get
\begin{equation}\label{2c13}\\\begin{cases}
-\Delta v_2^R(x)=f_{2}(u)(x),\ \ x\in B_R(0),\\
v_2^R(x)=0,\ \ \ \ \ \ \ x\in \mathbb{R}^{n}\setminus B_R(0).
\end{cases}\end{equation}
Let $w_2^R(x):=(-\Delta)^{m-2+\frac{\alpha}{2}}u(x)-v_2^R(x)$. By Theorem \ref{Thm0}, \eqref{2c11} and \eqref{2c13}, we have
\begin{equation}\label{2c14}\\\begin{cases}
-\Delta w_2^R(x)=0,\ \ \ \ x\in B_R(0),\\
w_2^R(x)\geq0, \,\,\,\,\, x\in \mathbb{R}^{n}\setminus B_R(0).
\end{cases}\end{equation}
By maximum principle, we deduce that for any $R>0$,
\begin{equation}\label{2c15}
  w_2^R(x)=(-\Delta)^{m-2+\frac{\alpha}{2}}u(x)-v_2^{R}(x)\geq0, \,\,\,\,\,\,\, \forall \,\, x\in\mathbb{R}^{n}.
\end{equation}
Now, for each fixed $x\in\mathbb{R}^{n}$, letting $R\rightarrow\infty$ in \eqref{2c15}, we have
\begin{equation}\label{2c16}
(-\Delta)^{m-2+\frac{\alpha}{2}}u(x)\geq\int_{\mathbb{R}^{n}}\frac{R_{2,n}}{|x-y|^{n-2}}f_{2}(u)(y)dy=:v_{2}(x)\geq0.
\end{equation}
Take $x=0$ in \eqref{2c16}, we get
\begin{equation}\label{2c17}
  \int_{\mathbb{R}^{n}}\frac{C_{1}}{|y|^{n-2}}dy\leq\int_{\mathbb{R}^{n}}\frac{f_{2}(u)(y)}{|y|^{n-2}}dy<+\infty,
\end{equation}
it follows easily that $C_{1}=0$, and hence we have proved \eqref{3-1}, that is,
\begin{equation}\label{2c18}
  u_{m}(x)=(-\Delta)^{m-1+\frac{\alpha}{2}}u(x)=f_{2}(u)(x)
  =\int_{\mathbb{R}^{n}}\frac{R_{2,n}}{|x-y|^{n-2}}\left(\frac{1}{|\cdot|^{\sigma}}\ast|u|^{p}\right)(y)u^{q}(y)dy.
\end{equation}

One can easily observe that $v_{2}$ is a solution of
\begin{equation}\label{2c19}
-\Delta v_{2}(x)=f_{2}(u)(x),  \,\,\,\,\, x\in \mathbb{R}^n.
\end{equation}
Define $w_{2}(x):=(-\Delta)^{m-2+\frac{\alpha}{2}}u(x)-v_{2}(x)$, then it satisfies
\begin{equation}\label{2c20}\\\begin{cases}
-\Delta w_{2}(x)=0, \,\,\,\,\,  x\in \mathbb{R}^n,\\
w_{2}(x)\geq0, \,\,\,\,\,\,  x\in\mathbb{R}^n.
\end{cases}\end{equation}
From Liouville theorem for harmonic functions, we can deduce that
\begin{equation}\label{2c21}
   w_{2}(x)=(-\Delta)^{m-2+\frac{\alpha}{2}}u(x)-v_{2}(x)\equiv C_{2}\geq0.
\end{equation}
Therefore, we have proved that
\begin{equation}\label{2c22}
  (-\Delta)^{m-2+\frac{\alpha}{2}}u(x)=\int_{\mathbb{R}^{n}}\frac{R_{2,n}}{|x-y|^{n-2}}f_{2}(u)(y)dy+C_{2}=:f_{3}(u)(x)\geq C_{2}\geq0.
\end{equation}
By the same methods as above, we can prove that
\begin{equation}\label{3-2}
(-\Delta)^{m-3+\frac{\alpha}{2}}u(x)\geq\int_{\mathbb{R}^{n}}\frac{R_{2,n}}{|x-y|^{n-2}}f_{3}(u)(y)dy=:v_{3}(x)\geq0.
\end{equation}
Take $x=0$ in \eqref{3-2}, we get
\begin{equation}\label{3-3}
  \int_{\mathbb{R}^{n}}\frac{C_{2}}{|y|^{n-2}}dy\leq\int_{\mathbb{R}^{n}}\frac{f_{3}(u)(y)}{|y|^{n-2}}dy<+\infty,
\end{equation}
it follows easily that $C_{2}=0$, and hence we have
\begin{equation}\label{2c23}
  u_{m-1}(x)=(-\Delta)^{m-2+\frac{\alpha}{2}}u(x)=f_{3}(u)(x)
  =\int_{\mathbb{R}^{n}}\frac{R_{2,n}}{|x-y|^{n-2}}f_{2}(u)(y)dy.
\end{equation}
Repeating the above arguments, defining
\begin{equation}\label{2c24}
  f_{k+1}(u)(x):=\int_{\mathbb{R}^{n}}\frac{R_{2,n}}{|x-y|^{n-2}}f_{k}(u)(y)dy
\end{equation}
for $k=1,2,\cdots,m-1$, then by Theorem \ref{Thm0} and induction, we have
\begin{equation}\label{2c25}
  u_{m+1-k}(x)=(-\Delta)^{m-k+\frac{\alpha}{2}}u(x)=f_{k+1}(u)(x)=\int_{\mathbb{R}^{n}}\frac{R_{2,n}}{|x-y|^{n-2}}f_{k}(u)(y)dy
\end{equation}
for $k=1,2,\cdots,m-1$, and
\begin{equation}\label{2c50}
  u_{1}(x)=(-\Delta)^{\frac{\alpha}{2}}u(x)=\int_{\mathbb{R}^{n}}\frac{R_{2,n}}{|x-y|^{n-2}}f_{m}(u)(y)dy+C_{m}=:f_{m+1}(u)(x)\geq C_{m}\geq0.
\end{equation}
For arbitrary $R>0$, let
\begin{equation}\label{2c12+}
v_{m+1}^R(x):=\int_{B_R(0)}G^{\alpha}_R(x,y)f_{m+1}(u)(y)dy,
\end{equation}
where, in the cases $0<\alpha<2$, the Green's function for $(-\Delta)^{\frac{\alpha}{2}}$ on $B_R(0)$ is given by
\begin{equation}\label{2-8c+}
G^\alpha_R(x,y):=\frac{C_{n,\alpha}}{|x-y|^{n-\alpha}}\int_{0}^{\frac{t_{R}}{s_{R}}}\frac{b^{\frac{\alpha}{2}-1}}{(1+b)^{\frac{n}{2}}}db
\,\,\,\,\,\,\,\,\, \text{if} \,\, x,y\in B_{R}(0)
\end{equation}
with $s_{R}=\frac{|x-y|^{2}}{R^{2}}$, $t_{R}=\left(1-\frac{|x|^{2}}{R^{2}}\right)\left(1-\frac{|y|^{2}}{R^{2}}\right)$, and $G^{\alpha}_{R}(x,y)=0$ if $x$ or $y\in\mathbb{R}^{n}\setminus B_{R}(0)$ (see \cite{K}). Then, we can get
\begin{equation}\label{2c13+}\\\begin{cases}
(-\Delta)^{\frac{\alpha}{2}}v_{m+1}^R(x)=f_{m+1}(u)(x),\ \ x\in B_R(0),\\
v_{m+1}^R(x)=0,\ \ \ \ \ \ \ x\in \mathbb{R}^{n}\setminus B_R(0).
\end{cases}\end{equation}
Let $w_{m+1}^R(x):=u(x)-v_{m+1}^R(x)$. By \eqref{2c50} and \eqref{2c13+}, we have
\begin{equation}\label{2c14+}\\\begin{cases}
(-\Delta)^{\frac{\alpha}{2}}w_{m+1}^R(x)=0,\ \ \ \ x\in B_R(0),\\
w_{m+1}^R(x)\geq0, \,\,\,\,\, x\in \mathbb{R}^{n}\setminus B_R(0).
\end{cases}\end{equation}

Now we need the following maximum principle for fractional Laplacians $(-\Delta)^{\frac{\alpha}{2}}$, which can been found in \cite{CLL,S}.
\begin{lem}\label{max}
Let $\Omega$ be a bounded domain in $\mathbb{R}^{n}$ and $0<\alpha<2$. Assume that $u\in\mathcal{L}_{\alpha}\cap C^{1,1}_{loc}(\Omega)$ and is l.s.c. on $\overline{\Omega}$. If $(-\Delta)^{\frac{\alpha}{2}}u\geq 0$ in $\Omega$ and $u\geq 0$ in $\mathbb{R}^n\setminus\Omega$, then $u\geq 0$ in $\mathbb{R}^n$. Moreover, if $u=0$ at some point in $\Omega$, then $u=0$ a.e. in $\mathbb{R}^{n}$. These conclusions also hold for unbounded domain $\Omega$ if we assume further that
\[\liminf_{|x|\rightarrow\infty}u(x)\geq0.\]
\end{lem}

By maximal principle for $-\Delta$ if $\alpha=2$ and Lemma \ref{max} if $0<\alpha<2$, we can deduce immediately from \eqref{2c14+} that for any $R>0$,
\begin{equation}\label{2c15+}
  w_{m+1}^R(x)=u(x)-v_{m+1}^{R}(x)\geq0, \,\,\,\,\,\,\, \forall \,\, x\in\mathbb{R}^{n}.
\end{equation}
Now, for each fixed $x\in\mathbb{R}^{n}$, letting $R\rightarrow\infty$ in \eqref{2c15+}, we have
\begin{equation}\label{2c16+}
u(x)\geq\int_{\mathbb{R}^{n}}\frac{R_{\alpha,n}}{|x-y|^{n-\alpha}}f_{m+1}(u)(y)dy=:v_{m+1}(x)\geq0.
\end{equation}
Take $x=0$ in \eqref{2c16+}, we get
\begin{equation}\label{2c17+}
  \int_{\mathbb{R}^{n}}\frac{C_{m}}{|y|^{n-\alpha}}dy\leq\int_{\mathbb{R}^{n}}\frac{f_{m+1}(u)(y)}{|y|^{n-\alpha}}dy<+\infty,
\end{equation}
it follows easily that $C_{m}=0$, and hence we have
\begin{equation}\label{2c18+}
  u_{1}(x)=(-\Delta)^{\frac{\alpha}{2}}u(x)=f_{m+1}(u)(x)=\int_{\mathbb{R}^{n}}\frac{R_{2,n}}{|x-y|^{n-2}}f_{m}(u)(y)dy.
\end{equation}

One can easily observe that $v_{m+1}$ is a solution of
\begin{equation}\label{3-4}
(-\Delta)^{\frac{\alpha}{2}}v_{m+1}(x)=f_{m+1}(x),  \,\,\,\,\,\,\, x\in \mathbb{R}^n.
\end{equation}
Define $w_{m+1}(x):=u(x)-v_{m+1}(x)$. Then, by \eqref{2c16+} and \eqref{2c18+}, we have $w_{m+1}$ satisfies
\begin{equation}\label{3-5}\\\begin{cases}
(-\Delta)^{\frac{\alpha}{2}}w_{m+1}(x)=0, \,\,\,\,\,  x\in \mathbb{R}^n,\\
w_{m+1}(x)\geq0, \,\,\,\,\,\,  x\in \mathbb{R}^n.
\end{cases}\end{equation}

Now we need the following Liouville theorem for fractional Laplacians $(-\Delta)^{\frac{\alpha}{2}}$, which can been found in \cite{BKN,ZCCY}.
\begin{lem}\label{Liouville}(Liouville theorem, \cite{BKN,ZCCY})
Assume $n\geq2$ and $0<\alpha<2$. Let $u$ be a strong solution of
\begin{equation*}\\\begin{cases}
	(-\Delta)^{\frac{\alpha}{2}}u(x)=0, \,\,\,\,\,\,\,\, x\in\mathbb{R}^{n}, \\
	u(x)\geq0, \,\,\,\,\,\,\, x\in\mathbb{R}^{n},
\end{cases}\end{equation*}
then $u\equiv C\geq0$.
\end{lem}

From Liouville theorem for harmonic functions if $\alpha=2$ and Lemma \ref{Liouville} if $0<\alpha<2$, we can deduce that
\begin{equation}\label{3-6}
   w_{m+1}(x)=u(x)-v_{m+1}(x)\equiv C_{m+1}\geq0.
\end{equation}
Therefore, we have
\begin{equation}\label{3-7}
  u(x)=\int_{\mathbb{R}^{n}}\frac{R_{\alpha,n}}{|x-y|^{n-\alpha}}f_{m+1}(u)(y)dy+C_{m+1}\geq C_{m+1}\geq0.
\end{equation}
Consequently, we get
\begin{equation}\label{3-8}
  \int_{\mathbb{R}^{n}}\frac{(C_{m+1})^{p}}{|x|^{\sigma}}dx\leq\int_{\mathbb{R}^{n}}\frac{u^{p}(x)}{|x|^{\sigma}}dx<+\infty,
\end{equation}
which implies $C_{m+1}=0$, and hence
\begin{eqnarray}\label{3-9}
  u(x)&=&\int_{\mathbb{R}^{n}}\frac{R_{\alpha,n}}{|x-y|^{n-\alpha}}f_{m+1}(u)(y)dy=\int_{\mathbb{R}^{n}}\frac{R_{\alpha,n}}{|x-y^{m+1}|^{n-\alpha}}
  \int_{\mathbb{R}^{n}}\frac{R_{2,n}}{|y^{m+1}-y^{m}|^{n-2}} \\
  &&\cdots\int_{\mathbb{R}^{n}}\frac{R_{2,n}}{|y^{2}-y^{1}|^{n-2}}\left(\frac{1}{|\cdot|^{\sigma}}\ast u^{p}\right)(y^{1})u^{q}(y^{1})dy^{1}\cdots dy^{m}dy^{m+1}. \nonumber
\end{eqnarray}

From the properties of Riesz potential, we have the following formula (see \cite{Stein}), that is, for any $\alpha_{1},\alpha_{2}\in(0,n)$ such that $\alpha_{1}+\alpha_{2}\in(0,n)$, one has
\begin{equation}\label{3-27}
\int_{\mathbb{R}^{n}}\frac{R_{\alpha_{1},n}}{|x-y|^{n-\alpha_{1}}}\cdot\frac{R_{\alpha_{2},n}}{|y-z|^{n-\alpha_{2}}}dy
=\frac{R_{\alpha_{1}+\alpha_{2},n}}{|x-z|^{n-(\alpha_{1}+\alpha_{2})}}.
\end{equation}
Now, by applying the formula \eqref{3-27} and direct calculations to \eqref{3-9}, we obtain that
\begin{equation}\label{3-28}
u(x)=\int_{\mathbb{R}^n}\frac{R_{2s,n}}{|x-y|^{n-2s}}\left(\int_{\mathbb{R}^n}\frac{u^{p}(z)}{|y-z|^{\sigma}}dz\right)u^{q}(y)dy,
\end{equation}
that is, $u$ satisfies the integral equation \eqref{IE}. This concludes our proof of Theorem \ref{equivalence}.

\section{Proof of Theorem \ref{Thm}}
In this section, we will apply the method of moving of spheres to integral equations \eqref{IE} and establish the classification results of nonnegative solutions, that is, Theorem \ref{Thm}.

Assume $n\geq1$, $0<s:=m+\frac{\alpha}{2}<\frac{n}{2}$, $0<\alpha\leq2$, $m\geq0$ is an integer, $0<\sigma<n$, $0<p\leq\frac{2n-\sigma}{n-2s}$ and $0<q\leq\frac{n+2s-\sigma}{n-2s}$. Let $u\in C(\mathbb{R}^{n})$ be a nonnegative solution to IE \eqref{IE}. It follows from integral equation \eqref{IE} and ``bootstrap" methods that $u\in C^{\infty}(\mathbb{R}^{n})$. In the following, we assume $u$ is a non-trivial nonnegative solutions (i.e., $u\not\equiv 0$) to IE \eqref{IE}. Then, integral equation \eqref{IE} implies that $u>0$ in $\mathbb{R}^{n}$ and $\int_{\mathbb{R}^{n}}\frac{u^{p}(x)}{|x|^{\sigma}}dx<+\infty$. We will derive the unique form (up to translations and scalings) of the positive solution $u$ in the critical case $p=\frac{2n-\sigma}{n-2s}$ and $q=\frac{n+2s-\sigma}{n-2s}$, and obtain a contradiction and hence the nonexistence of positive solutions in the subcritical cases $0<p<\frac{2n-\sigma}{n-2s}$ or $0<q<\frac{n+2s-\sigma}{n-2s}$.

Set $P(y)=\int_{\mathbb{R}^n}\frac{u^{p}(z)}{|y-z|^{\sigma}}dz$, then \eqref{IE} can be written as
\begin{equation}\label{4-1}
  u(x)=\int_{\mathbb{R}^n}\frac{R_{2s,n}}{|x-y|^{n-2s}}P(y)u^{q}(y)dy.
\end{equation}
For arbitrary $x_0\in \mathbb{R}^n$ and $\lambda>0$, we define the Kelvin transforms centered at $x_{0}$ by
\begin{equation}\label{4-2}
u_{x_0,\lambda}(x):={\left(\frac{\lambda}{|x-x_0|}\right)}^{n-2s}u(x^\lambda), \qquad \forall \,\, x\in \mathbb{R}^{n}\setminus\{x_0\},
\end{equation}
where $x^\lambda=\frac{\lambda^2(x-x_0)}{|x-x_0|^2}+x_0$. Then, by \eqref{4-1} and \eqref{4-2}, we derive
\begin{align}\label{4-3}
u_{x_0,\lambda}(x)&={\left(\frac{\lambda}{|x-x_0|}\right)}^{n-2s}u(x^\lambda)\\
&=\int_{\mathbb{R}^n} \frac{R_{2s,n}P_{x_0,\lambda}(y)u_{x_0,\lambda}^{q}(y)}{|x-y|^{n-2s}}{\left(\frac{\lambda}{|y-x_0|}\right)}^\tau dy, \nonumber
\end{align}
where $\tau:=(n+2s-\sigma)-q(n-2s)\geq0$ and
\begin{equation}\label{4-3'}
  P_{x_0,\lambda}(y):=\left(\frac{\lambda}{|y-x_{0}|}\right)^{\sigma}P(y^{\lambda})=\int_{\mathbb{R}^n}\frac{u_{x_0,\lambda}^{p}(z)}{|y-z|^{\sigma}}
  \left(\frac{\lambda}{|z-x_{0}|}\right)^{\mu}dz
\end{equation}
with $\mu:=(2n-\sigma)-p(n-2s)\geq0$.

By \eqref{4-1}, \eqref{4-3} and straightforward calculations, one can verify
\begin{align}\label{4-4}
&u(x)=\int_{\mathbb{R}^n} \frac{R_{2s,n}}{|x-y|^{n-2s}}P(y)u^{q}(y)dy=\int_{B_\lambda(x_0)} \frac{R_{2s,n}P(y)u^{q}(y)}{|x-y|^{n-2s}}dy \\
&\qquad\quad +\int_{B_\lambda(x_0)} \frac{R_{2s,n}P(y^\lambda)}{{\left|\frac{|y-x_0|(x-x_0)}{\lambda}-\frac{\lambda (y-x_0)}{|y-x_0|}\right|}^{n-2s}}{\left( \frac{\lambda}{|y-x_0|}\right)}^{n+2s} u^{q}(y^\lambda)dy \nonumber \\
&\qquad=\int_{B_\lambda(x_0)} \frac{R_{2s,n}P(y)u^{q}(y)}{|x-y|^{n-2s}}dy \nonumber \\
&\qquad\quad +\int_{B_\lambda(x_0)} \frac{R_{2s,n}P_{x_0,\lambda}(y)}{{\left|\frac{|y-x_0|(x-x_0)}{\lambda}-\frac{\lambda (y-x_0)}{|y-x_0|}\right|}^{n-2s}}{\left( \frac{\lambda}{|y-x_0|}\right)}^{\tau}u_{x_0,\lambda}^{q}(y)dy, \nonumber
\end{align}
and
\begin{align}\label{4-5}
& u_{x_0,\lambda}(x)=\int_{\mathbb{R}^n} \frac{R_{2s,n}P_{x_0,\lambda}(y)u_{x_0,\lambda}^{q}(y)}{|x-y|^{n-2s}}{\left( \frac{\lambda}{|y-x_0|}\right)}^\tau dy \\
&\qquad\quad\,\, =\int_{B_\lambda(x_0)} \frac{R_{2s,n}P_{x_0,\lambda}(y)}{|x-y|^{n-2s}}{\left( \frac{\lambda}{|y-x_0|}\right)}^\tau u_{x_0,\lambda}^{q}(y)dy \nonumber \\
&\qquad\qquad\,\, +\int_{B_\lambda(x_0)} \frac{R_{2s,n}P(y)}{{\left|\frac{|y-x_0|(x-x_0)}{\lambda}-\frac{\lambda (y-x_0)}{|y-x_0|}\right|}^{n-2s}}u^{q}(y)dy. \nonumber
\end{align}

Define $\omega_\lambda(x)=u_{x_0,\lambda}(x)-u(x)$ for $x\in \mathbb{R}^n \setminus\{x_0\}$. We first show that for $\lambda>0$ sufficiently small,
\begin{equation}\label{4-17}
  \omega_\lambda(x)\geq 0, \qquad \forall \,\, x\in B_\lambda(x_0)\setminus\{x_0\}.
\end{equation}
It suffices to show that $B_\lambda^{-}(x_0)=\emptyset$ for $\lambda$ small enough, where the set $B_\lambda^{-}(x_0)$ is defined by
\begin{equation}\label{4-18}
  B_\lambda^{-}(x_0):=\{x\in B_\lambda(x_0)\setminus\{x_0\}|\, \omega_\lambda(x)<0\}.
\end{equation}
Let us define
\begin{equation}\label{4-19}
  K_1^\lambda(x,y)=R_{2s,n}\left(\frac{1}{|x-y|^{n-2s}}-\frac{1}{{\Big|\frac{|y-x_0|(x-x_0)}{\lambda}-\frac{\lambda (y-x_0)}{|y-x_0|}\Big|}^{n-2s}}\right),
\end{equation}
\begin{equation}\label{4-20}
  K_2^\lambda(x,y)=\frac{1}{|x-y|^{\sigma}}-\frac{1}{{\Big|\frac{|y-x_0|(x-x_0)}{\lambda}-\frac{\lambda (y-x_0)}{|y-x_0|}\Big|}^{\sigma}}.
\end{equation}
One can easily verify that $K_1^\lambda(x,y)>0$ and $K_2^\lambda(x,y)>0$ for any $x, y\in B_\lambda(x_0)$.

Then, by \eqref{4-4} and \eqref{4-5}, for any $x\in B_\lambda(x_0)\setminus\{x_{0}\}$, we derive
\begin{align}\label{4-21}
&\quad \omega_\lambda(x)=u_{x_0,\lambda}(x)-u(x)\\
&=\int_{B_\lambda(x_0)}K_1^\lambda(x,y)\left[P_{x_0,\lambda}(y){\left( \frac{\lambda}{|y-x_0|}\right)}^\tau u_{x_0,\lambda}^{q}(y)-P(y)u^{q}(y)\right]dy \nonumber\\
&\geq\int_{Q_\lambda(x_0)}K_1^\lambda(x,y)\left[P_{x_0,\lambda}(y) u_{x_0,\lambda}^{q}(y)-P(y)u^{q}(y)\right]dy \nonumber\\
&=\int_{Q_\lambda(x_0)}K_1^\lambda(x,y)P(y)\left[u_{x_0,\lambda}^{q}(y)-u^{q}(y)\right]dy \nonumber\\
&\quad +\int_{Q_\lambda(x_0)}K_1^\lambda(x,y)\left[P_{x_0,\lambda}(y)-P(y)\right]u_{x_0,\lambda}^{q}(y)dy,\nonumber
\end{align}
where the subset $Q_\lambda(x_0)\subseteq B_\lambda(x_0)\setminus\{x_{0}\}$ is defined by
\begin{equation}\label{4-28}
  Q_\lambda(x_0):=\{x\in B_\lambda(x_0)\setminus\{x_{0}\}|\,P_{x_{0},\lambda}(x)u_{x_{0},\lambda}^{q}(x)<P(x)u^{q}(x)\}.
\end{equation}
Note that through direct calculations, one has, for any $y\in B_\lambda(x_0)$,
\begin{eqnarray}\label{4-22}
  P_{x_0,\lambda}(y)-P(y)&=&\int_{B_\lambda(x_0)}K_2^\lambda(y,z)\left[\left(\frac{\lambda}{|z-x_{0}|}\right)^{\mu}u_{x_0,\lambda}^{p}(z)-u^{p}(z)\right]dz \\
  &\geq& \int_{B_\lambda(x_0)}K_2^\lambda(y,z)\left[u_{x_0,\lambda}^{p}(z)-u^{p}(z)\right]dz. \nonumber
\end{eqnarray}
Thus from \eqref{4-21}, \eqref{4-22} and mean value theorem, we get, for any $x\in B_\lambda^{-}(x_0)$,
\begin{align}\label{4-6}
&0>\omega_\lambda(x)\geq\int_{Q_\lambda(x_0)} K_1^\lambda(x,y)P(y)\left[u_{x_0,\lambda}^{q}(y)-u^{q}(y)\right]dy \\
&\quad\,\,+\int_{Q_\lambda(x_0)}K_1^\lambda(x,y)\left(\int_{B_\lambda(x_0)}K_2^\lambda(y,z)\left[u_{x_0,\lambda}^{p}(z)-u^{p}(z)\right]dz\right)
u_{x_0,\lambda}^{q}(y)dy \nonumber\\
&\,\,\,\geq q\int_{Q_\lambda(x_0)\cap B_\lambda^-(x_0)}K_1^\lambda(x,y)P(y)\max\{u^{q-1}(y),u_{x_0,\lambda}^{q-1}(y)\}\omega_\lambda(y)dy \nonumber\\
&\quad\,\,+p\int_{Q_\lambda(x_0)}K_1^\lambda(x,y)\left(\int_{B_\lambda^-(x_0)}K_2^\lambda(y,z)\max\{u^{p-1}(z),u_{x_0,\lambda}^{p-1}(z)\}\omega_\lambda(z)dz\right)u_{x_0,\lambda}^{q}(y)dy \nonumber\\
&\,\,\,\geq q\int_{Q_\lambda(x_0)\cap B_\lambda^-(x_0)}\frac{R_{2s,n}P(y)\omega_\lambda(y)}{|x-y|^{n-2s}}\max\{u^{q-1}(y),u_{x_0,\lambda}^{q-1}(y)\}dy \nonumber\\
&\quad\,\,+p\int_{Q_\lambda(x_0)}\frac{R_{2s,n}}{|x-y|^{n-2s}}\left(\int_{B_\lambda^-(x_0)}\frac{\omega_\lambda(z)}{|y-z|^{\sigma}}\max\{u^{p-1}(z),u_{x_0,\lambda}^{p-1}(z)\}dz\right)
u_{x_0,\lambda}^{q}(y)dy.\nonumber
\end{align}

Now we need the following Hardy-Littlewood-Sobolev inequality.
\begin{lem}\label{HL}(Hardy-Littlewood-Sobolev inequality)
Let $n\geq1$, $0<s<n$ and $1<p<q<\infty$ be such that $\frac{n}{q}=\frac{n}{p}-s$. Then we have
\begin{equation}\label{HLS}
	\Big\|\int_{\mathbb{R}^{n}}\frac{f(y)}{|x-y|^{n-s}}dy\Big\|_{L^{q}(\mathbb{R}^{n})}\leq C_{n,s,p,q}\|f\|_{L^{p}(\mathbb{R}^{n})},
\end{equation}
for all $f\in L^{p}(\mathbb{R}^{n})$.
\end{lem}

Set $U_{\lambda}(x):=\max\{u^{q-1}(x),u_{x_0,\lambda}^{q-1}(x)\}$, $V_{\lambda}(x):=\max\{u^{p-1}(x),u_{x_0,\lambda}^{p-1}(x)\}$, and
\begin{equation}\label{4-23}
  W_\lambda(y):=\int_{B_\lambda^-(x_0)} \frac{V_{\lambda}(z)\omega_\lambda(z)}{|y-z|^{\sigma}}dz.
\end{equation}
By \eqref{4-6}, Hardy-Littlewood-Sobolev inequality and H\"{o}lder inequality, we have, for any $r>max\{\frac{n}{n-2s},\frac{n}{\sigma}\}$,
\begin{align}\label{4-7}
&\quad \|\omega_\lambda\|_{L^{r}(B_\lambda^-(x_0))} \\
&\leq C{\left\|PU_{\lambda}\omega_\lambda\right\|}_{L^{\frac{nr}{n+2sr}}(Q_{\lambda}(x_{0})\cap B_\lambda^-(x_0))}+C{\left\|W_\lambda u_{x_0,\lambda}^{q}\right\|}_{L^{\frac{nr}{n+2sr}}(Q_\lambda(x_0))} \nonumber\\
&\leq C{\left\|PU_{\lambda}\right\|}_{L^{\frac{n}{2s}}(B_\lambda^-(x_0))}{\left\|\omega_\lambda\right\|}_{L^r(B_\lambda^-(x_0))}+ C{\left\|u_{x_0,\lambda}^{q}\right\|}_{L^{\frac{n}{2s}}(Q_\lambda(x_0))} {\left\|W_\lambda\right\|}_{L^{r}(Q_\lambda(x_0))} \nonumber\\
&\leq C{\left\|PU_{\lambda}\right\|}_{L^{\frac{n}{2s}}(B_\lambda^-(x_0))}{\left\|\omega_\lambda\right\|}_{L^r(B_\lambda^-(x_0))}+ C{\left\| u_{x_0,\lambda}^{q}\right\|}_{L^{\frac{n}{2s}}(Q_\lambda(x_0))}{\left\|V_{\lambda}\omega_\lambda\right\|}_{L^{\frac{nr}{n+(n-\sigma)r}}(B_\lambda^-(x_0))} \nonumber\\
&\leq C{\left\|PU_{\lambda}\right\|}_{L^{\frac{n}{2s}}(B_\lambda^-(x_0))}{\left\|\omega_\lambda\right\|}_{L^r(B_\lambda^-(x_0))} \nonumber \\
&\quad +C{\left\|u_{x_0,\lambda}^{q}\right\|}_{L^{\frac{n}{2s}}(Q_\lambda(x_0))}\left\|V_{\lambda}\right\|_{L^{\frac{n}{n-\sigma}}(B_\lambda^-(x_0))}
\|\omega_\lambda\|_{L^r(B_\lambda^-(x_0))} \nonumber\\
&\leq C\left({\left\|PU_{\lambda}\right\|}_{L^{\frac{n}{2s}}(B_\lambda^{-}(x_0))}+\left\|u_{x_0,\lambda}^{q}\right\|_{L^{\frac{n}{2s}}(Q_\lambda(x_0))} \left\|V_{\lambda}\right\|_{L^{\frac{n}{n-\sigma}}(B_\lambda^-(x_0))}\right)\left\|\omega_\lambda\right\|_{L^r(B_\lambda^-(x_0))}.\nonumber
\end{align}

Integral equation \eqref{IE} implies that, there exists a constant $C>0$, such that the positive solution $u$ satisfies the following lower bound:
\begin{equation}\label{4-8}
u(x)\geq\frac{C}{|x-x_0|^{n-2s}}, \,\,\,\,\,\,\,\,\,\,\,\,\, \forall \,\,\, |x-x_0|\geq1.
\end{equation}
Indeed, since $u>0$ solves the integral equation \eqref{IE}, we can infer that, for any $|x-x_0|\geq1$,
\begin{eqnarray}\label{4-24}
	u(x)&\geq&\int_{|y-x_0|\leq\frac{1}{2}} \frac{1}{|x-y|^{n-2s}} \int_{|z-x_0|>1} \frac{u^{p}(z)}{|y-z|^{\sigma}}dz u^{q}(y)dy\\
	\nonumber &\geq& \frac{C}{|x-x_0|^{n-2s}}\int_{|y-x_0|\leq\frac{1}{2}}\int_{|z-x_0|>1}\frac{u^{p}(z)}{|z-x_0|^{\sigma}}dz u^{q}(y) dy=:\frac{C}{|x-x_0|^{n-2s}}. \nonumber
\end{eqnarray}

As a consequence, we can obtain the following lower bounds for $u_{x_0,\lambda}$ in $B_\lambda(x_0)\setminus\{x_0\}$. Let us consider $0<\lambda<1$ first. If $0<|x-x_0|\leq\lambda^2$, then $|x^\lambda-x_0|\geq1$, by \eqref{4-8}, we have
\begin{equation}\label{4-25}
  u_{x_0,\lambda}(x)={\left(\frac{\lambda}{|x-x_0|}\right)}^{n-2s}u(x^\lambda)\geq\frac{C}{\lambda^{n-2s}}.
\end{equation}
If $\lambda^2\leq|x-x_0|<\lambda$, then $|x^\lambda-x_0|\leq1$, we have
\begin{equation}\label{4-26}
  u_{x_0,\lambda}(x)={\left(\frac{\lambda}{|x-x_0|}\right)}^{n-2s}u(x^\lambda)\geq\left(\min_{B_1(x_0)} u\right){\left(\frac{\lambda}{|x-x_0|}\right)}^{n-2s}\geq  \min_{B_1(x_0)}u=:C>0.
\end{equation}
For $\lambda\geq1$, one has $|x^\lambda-x_0|>\lambda\geq1$ for any $x\in B_\lambda(x_0)\setminus\{x_0\}$, and hence \eqref{4-8} yields
\begin{equation}\label{4-29}
  u_{x_0,\lambda}(x)={\left(\frac{\lambda}{|x-x_0|}\right)}^{n-2s}u(x^\lambda)\geq \frac{C}{\lambda^{n-2s}}, \qquad \forall \,\, x\in B_\lambda(x_0)\setminus\{x_0\}.
\end{equation}

From the definition of $P(y)$, one can get the following lower bound:
\begin{equation}\label{4-30}
  P(y)\geq \frac{C}{|y-x_{0}|^{\sigma}}\int_{|z-x_{0}|\leq\frac{1}{2}}u^{p}(z)dz=:\frac{C}{|y-x_{0}|^{\sigma}}, \qquad \forall \,\, |y-x_{0}|\geq1.
\end{equation}
Consequently, we can obtain the following lower bounds for $P_{x_0,\lambda}$ in $B_\lambda(x_0)\setminus\{x_0\}$. Let us consider $0<\lambda<1$ first. If $0<|y-x_0|\leq\lambda^2$, then $|y^\lambda-x_0|\geq1$, by \eqref{4-30}, we have
\begin{equation}\label{4-31}
  P_{x_0,\lambda}(y)={\left(\frac{\lambda}{|y-x_0|}\right)}^{\sigma}P(y^\lambda)\geq \frac{C}{\lambda^{\sigma}}.
\end{equation}
If $\lambda^2\leq|y-x_0|<\lambda$, then $|y^\lambda-x_0|\leq1$, we have
\begin{equation}\label{4-32}
  P_{x_0,\lambda}(y)={\left(\frac{\lambda}{|y-x_0|}\right)}^{\sigma}P(y^\lambda)\geq\left(\min_{B_1(x_0)} P\right){\left(\frac{\lambda}{|y-x_0|}\right)}^{\sigma}\geq\min_{B_1(x_0)}P=:C>0.
\end{equation}
For $\lambda\geq1$, one has $|y^\lambda-x_0|>\lambda\geq1$ for any $y\in B_\lambda(x_0)\setminus\{x_0\}$, and hence \eqref{4-30} yields
\begin{equation}\label{4-33}
  P_{x_0,\lambda}(y)={\left(\frac{\lambda}{|y-x_0|}\right)}^{\sigma}P(y^\lambda)\geq\frac{C}{\lambda^{\sigma}}, \qquad \forall \,\, y\in B_\lambda(x_0)\setminus\{x_0\}.
\end{equation}

From the lower bounds \eqref{4-25} and \eqref{4-26} of $u_{x_0,\lambda}$ in $B_\lambda(x_0)\setminus\{x_0\}$ and the continuity of $u$ and $P$ in $\mathbb{R}^{n}$, we can infer that
\begin{equation}\label{4-27}
  {\left\|PU_{\lambda}\right\|}_{L^{\frac{n}{2s}}(B_\lambda^-(x_0))} \rightarrow 0, \qquad \text{as}\,\,\lambda \rightarrow 0.
\end{equation}
By the definition of $Q_{\lambda}(x_{0})$, the continuity of $u$ and $P$ in $\mathbb{R}^{n}$, the lower bounds \eqref{4-31} and \eqref{4-32}, we can get the following upper bounds: for any $x\in Q_{\lambda}(x_{0})$,
\begin{equation}\label{4-34}
  u_{x_{0},\lambda}^{q}(x)<\frac{P(x)u^{q}(x)}{P_{x_{0},\lambda}(x)}\leq C(1+\lambda^{\sigma})\left(\max_{B_{1}(x_{0})}Pu^{q}\right)=:C(1+\lambda^{\sigma}), \qquad \text{if} \,\, 0<\lambda<1;
\end{equation}
\begin{equation}\label{4-35}
  u_{x_{0},\lambda}^{q}(x)<\frac{P(x)u^{q}(x)}{P_{x_{0},\lambda}(x)}\leq C(1+\lambda^{\sigma})\left(\max_{B_{\lambda}(x_{0})}Pu^{q}\right)=:C_{\lambda}(1+\lambda^{\sigma}), \qquad \text{if} \,\, \lambda\geq1.
\end{equation}
From the lower bounds \eqref{4-25} and \eqref{4-26} of $u_{x_0,\lambda}$ in $B_\lambda(x_0)\setminus\{x_0\}$, the upper bound \eqref{4-34} of $u_{x_{0},\lambda}^{q}$ in $Q_{\lambda}(x_{0})$ and the continuity of $u$ in $\mathbb{R}^{n}$, we deduce that
\begin{equation}\label{4-36}
  {\left\|u_{x_0,\lambda}^{q}\right\|}_{L^{\frac{n}{2s}}(Q_\lambda(x_0))}{\left\|V_{\lambda}\right\|}_{L^{\frac{n}{n-\sigma}}(B_\lambda^-(x_0))}\rightarrow 0, \qquad \text{as}\,\,\lambda \rightarrow 0.
\end{equation}
Thus, there exists $0<\epsilon_0<1$ small enough such that for all $0<\lambda<\epsilon_0$,
\begin{equation}\label{4-37}
  C\left({\left\|PU_{\lambda}\right\|}_{L^{\frac{n}{2s}}(B_\lambda^-(x_0))}+{\left\|u_{x_0,\lambda}^{q}\right\|}_{L^{\frac{n}{2s}}(Q_\lambda(x_0))}
  {\left\|V_{\lambda}\right\|}_{L^{\frac{n}{n-\sigma}}(B_\lambda^-(x_0))}\right)<\frac{1}{2},
\end{equation}
where $C$ is the constant in the last inequality of \eqref{4-7}. Immediately, we conclude from \eqref{4-7} that $\|\omega_{\lambda}\|_{L^{r}(B_\lambda^-(x_0))}=0$, and thus $B_\lambda^-(x_0)$ has measure $0$. Then, $B_\lambda^-(x_0)=\emptyset$ must hold true for any $0<\lambda<\epsilon_0$ due to the continuity of $\omega_\lambda$ in $B_\lambda(x_0)\setminus\{x_0\}$, and hence we have derived \eqref{4-17} for $0<\lambda<\epsilon_0$.

This provides a starting point to carry out the method of moving spheres for arbitrarily given center $x_0\in \mathbb{R}^n$. Next, we will continuously increase the radius $\lambda$ as long as $\omega_{\lambda}\geq 0$ in $B_\lambda(x_0)\setminus\{x_0\}$ holds true. For a given center $x_0$, the critical scale $\lambda_{x_0}$ is defined by
\begin{equation}\label{4-9}
\lambda_{x_0}:=\sup\{\lambda>0|\,\omega_\mu\geq 0 \,\, \text{in} \,\, B_\mu(x_0)\setminus\{x_0\}, \, \forall \,0<\mu\leq\lambda\}>0.
\end{equation}

For the critical scale $\lambda_{x_0}$, we have the following crucial Lemma.
\begin{lem} \label{lem1}
One of the following two assertions holds, that is, either\\
\emph{(A)} For every $x_0\in\mathbb{R}^n$, the corresponding critical scale $\lambda_{x_0}>0$ is finite, or\\
\emph{(B)} For every $x_0\in\mathbb{R}^n$, the corresponding critical scale $\lambda_{x_0}=\infty$, this is, for every $\lambda>0$,  $\omega_\lambda\geq 0$ in $B_\lambda(x_0)\setminus\{x_0\}$.
\end{lem}

In order to prove Lemma \ref{lem1}, we need the following proposition.
\begin{prop}\label{prop}
If $\lambda_{x_0}<+\infty$, then $u_{x_0,\lambda_{x_0}}(x)\equiv u(x)$ for any $x\in B_{\lambda_{x_0}}(x_0) \setminus\{x_0\}$.
\end{prop}
\begin{proof}[Proof of Proposition \ref{prop}]
Proposition \ref{prop} will be proved by contradiction arguments. By the definition of $\lambda_{x_0}$, we know that $\omega_{\lambda_{x_0}}\geq 0$ in $B_{\lambda_{x_0}}(x_0)\setminus\{x_0\}$. Suppose on the contrary that $\omega_{\lambda_{x_0}}\not\equiv 0$, then there exists $\bar{x}\in B_{\lambda_{x_0}}(x_0) \setminus\{x_0\}$ and $\delta>0$ such that $B_\delta(\bar x)\subset B_{\lambda_{x_0}}(x_0) \setminus\{x_0\}$ and $\omega_{\lambda_{x_0}}>C>0$ in $B_\delta(\bar x)$. Then, by \eqref{4-21} and \eqref{4-22}, we have for any $x\in B_{\lambda_{x_0}}(x_0)\setminus\{x_0\}$,
\begin{align}\label{4-38}
	\omega_{\lambda_{x_0}}(x)&\geq\int_{B_{\lambda_{x_0}}(x_0)}K_1^{\lambda_{x_{0}}}(x,y)P(y)\left(u_{x_{0},\lambda_{x_{0}}}^{q}(y)-u^{q}(y)\right)dy \\
    &\geq q\int_{B_{\lambda_{x_0}}(x_0)}K_1^{\lambda_{x_{0}}}(x,y)P(y)\min\{u_{x_0,\lambda_{x_0}}^{q-1}(y), u^{q-1}(y)\}\omega_{\lambda_{x_{0}}}(y)dy \nonumber\\
	&\geq q\int_{B_{\delta}(\bar{x})}K_1^{\lambda_{x_{0}}}(x,y)P(y)\min\{u_{x_0,\lambda_{x_0}}^{q-1}(y), u^{q-1}(y)\}\omega_{\lambda_{x_{0}}}(y)dy \nonumber\\
	&>0. \nonumber
\end{align}
Let $\delta_1>0$ be sufficiently small, which will be determined later. Define the narrow region
\begin{equation}\label{4-10}
	A_{\delta_1}:=\{x\in \mathbb{R}^n \,|\,0<|x-x_0|<\delta_1\, \text{or}\,\lambda_{x_0}-\delta_1<|x-x_0|<\lambda_{x_0}\}.
\end{equation}	
Since $\omega_{\lambda_{x_0}}$ is continuous in $\mathbb{R}^{n}\setminus\{x_0\}$ and $A_{\delta_1}^c:=\left(B_{\lambda_{x_0}}(x_0)\setminus\{x_{0}\}\right)\setminus A_{\delta_1}$ is a compact subset, there exists a $C_0>0$ such that
\begin{equation}\label{4-11}
	\omega_{\lambda_{x_0}}(x)>C_0, \quad\quad \forall \,\, x\in A_{\delta_1}^c.
\end{equation}
By continuity, we can choose $\delta_2>0$ (depending on $\delta_{1}$) sufficiently small such that, for any $\lambda\in [\lambda_{x_0},\,\lambda_{x_0}+\delta_2]$,
\begin{equation}\label{4-12}
	\omega_{\lambda}(x)>\frac{C_0}{2}, \quad\quad \forall \,\, x\in A_{\delta_1}^c.
\end{equation}
Hence we must have, for any $\lambda\in [\lambda_{x_{0}},\,\lambda_{x_{0}}+\delta_2]$,
\begin{equation}\label{4-13}
	B_{\lambda}^-(x_0)\subset \left(B_{\lambda}(x_0)\setminus\{x_{0}\}\right)\setminus A^{c}_{\delta_1}=\{x\,|\,0<|x-x_0|<\delta_1\,\, \text{or}\,\,\lambda_{x_0}-\delta_1<|x-x_0|<\lambda\}.
\end{equation}
By \eqref{4-13}, the continuity of $u$ and $P$, the lower bounds \eqref{4-25}, \eqref{4-26} and \eqref{4-29} of $u_{x_{0},\lambda}$ in $B_{\lambda}(x_0)\setminus\{x_0\}$ and the upper bounds \eqref{4-34} and \eqref{4-35} of $u_{x_{0},\lambda}^{q}$ in $Q_{\lambda}(x_{0})$, we can choose $\delta_1$ sufficiently small (and $\delta_2$ more smaller if necessary) such that, for any $\lambda\in[\lambda_{x_{0}},\,\lambda_{x_{0}}+\delta_2]$,
\begin{equation}\label{4-14}
C\left({\left\|PU_{\lambda}\right\|}_{L^{\frac{n}{2s}}(B_\lambda^-(x_0))}+{\left\|u_{x_0,\lambda}^{q}\right\|}_{L^{\frac{n}{2s}}(Q_\lambda(x_0))} {\left\|V_{\lambda}\right\|}_{L^{\frac{n}{n-\sigma}}(B_\lambda^-(x_0))}\right)<\frac{1}{2},
\end{equation}
where $C$ is the constant in the last inequality of \eqref{4-7}. Immediately, we conclude from \eqref{4-7} that $\|\omega_{\lambda}\|_{L^{r}(B_\lambda^-(x_0))}=0$, and thus $B_\lambda^-(x_0)=\emptyset$ for any $\lambda\in[\lambda_{x_{0}},\,\lambda_{x_{0}}+\delta_2]$. Then, we obtain that for any $\lambda\in[\lambda_{x_{0}},\,\lambda_{x_{0}}+\delta_2]$,
\begin{equation}\label{4-15}
	\omega_{\lambda}(x)\geq 0, \quad\quad \forall \,\, x\in B_{\lambda}(x_0)\setminus\{x_0\},
\end{equation}
which contradicts the definition of critical scale $\lambda_{x_{0}}$. This finishes our proof of Proposition \ref{prop}.	
\end{proof}

\begin{proof}[Proof of Lemma \ref{lem1}]
Now we are ready to prove Lemma \ref{lem1}. If there exists a $x_0\in \mathbb{R}^n$ such that $\lambda_{x_0}=+\infty$, we have, for any $\lambda>0$,
\begin{equation}\label{4-39}
  \omega_\lambda(x)\geq 0, \qquad \forall \,\, x\in B_\lambda(x_0)\setminus\{x_0\},
\end{equation}
which implies that, for any $\lambda>0$,
\begin{equation}\label{4-40}
  u(x)\geq u_{x_{0},\lambda}(x), \qquad \forall \,\, |x-x_0|>\lambda.
\end{equation}
Then, due to the arbitrariness of $\lambda>0$, \eqref{4-40} yields that
\begin{equation}\label{4-41}
  \lim_{|x|\rightarrow+\infty}|x|^{n-2s}u(x)=+\infty.
\end{equation}
However, if we assume there exists another point $z_0\in \mathbb{R}^n$ such that $\lambda_{z_0}<+\infty$, then by Proposition \ref{prop}, we have
\begin{equation}\label{4-42}
  u_{z_{0},\lambda_{z_0}}(x)=u(x), \qquad \forall \,\, x\in\mathbb{R}^n\setminus\{z_0\}.
\end{equation}
The above identity immediately implies that
\begin{equation}\label{4-43}
  \lim_{|x|\rightarrow+\infty}|x|^{n-2s}u(x)=\left(\lambda_{z_{0}}\right)^{n-2s}u(z_{0})<+\infty.
\end{equation}
A contradiction! This concludes the proof of Lemma \ref{lem1}.
\end{proof}

In order to derive the classification results, we also need the following calculus Lemma (see Lemma 11.1 and Lemma 11.2 in \cite{LZ1}, see also \cite{Li,Xu}).
\begin{lem}(\cite{LZ1})\label{lem2}
Let $n\geq 1$, $\nu\in\mathbb{R}$ and $u\in C^{1}({\mathbb{R}^n})$. For every $x_0\in \mathbb{R}^n$ and $\lambda>0$, define $u_{x_0,\lambda}(x):={\left(\frac{\lambda}{|x-x_0|}\right)}^{\nu}u\left(\frac{\lambda^2(x-x_0)}{|x-x_0|^2}+x_0\right)$ for $x\in\mathbb{R}^{n}\setminus\{x_{0}\}$. Then, we have\\
\emph{(i)} If for every $x_0\in \mathbb{R}^n$, there exists a $0<\lambda_{x_0}<+\infty$ such that
\begin{equation*}
  u_{x_0,\lambda_{x_0}}(x)=u(x), \qquad \forall \,\, x\in\mathbb{R}^n\setminus\{x_0\},
\end{equation*}
then for some $C\in\mathbb{R}$, $\mu>0$ and $\bar x\in \mathbb{R}^n$,
\begin{equation*}
  u(x)=C{\left(\frac{\mu}{1+\mu^{2}|x-\bar x|^2}\right)}^{\frac{\nu}{2}}.
\end{equation*}
\emph{(ii)} If for every $x_0\in \mathbb{R}^n$ and $\lambda>0$,
\begin{equation*}
  u_{x_0,\lambda}(x)\geq u(x), \qquad \forall \,\, x\in B_\lambda(x_0)\setminus\{x_0\},
\end{equation*}
then $u\equiv C$ for some constant $C\in\mathbb{R}$.
\end{lem}
\begin{rem}\label{rem0}
In Lemma 11.1 and Lemma 11.2 of \cite{LZ1}, Li and Zhang have proved Lemma \ref{lem2} for $\nu>0$. Nevertheless, their methods can also be applied to show Lemma \ref{lem2} in the cases $\nu\leq0$, see \cite{Li,Xu}.
\end{rem}

In the subcritical cases $0<p<\frac{2n-\sigma}{n-2s}$ or $0<q<\frac{n+2s-\sigma}{n-2s}$, we have $\tau>0$ or $\mu>0$. Without loss of generality, suppose that $\tau>0$ and $\mu>0$. If there exists a $x_0\in\mathbb{R}^{n}$ such that the critical scale $\lambda_{x_0}<+\infty$, then by Proposition \ref{prop}, we have $\omega_{\lambda_{x_0}}(x)=u_{x_0,\lambda_{x_0}}(x)-u(x)\equiv 0$ for any $x\in\mathbb{R}^{n}\setminus\{x_{0}\}$. However, from \eqref{4-21}, \eqref{4-22}, $\tau>0$ and $\mu>0$, we deduce
\begin{align}\label{4-16}
&\quad \omega_{\lambda_{x_{0}}}(x)=u_{x_0,\lambda_{x_{0}}}(x)-u(x)\\
&=\int_{B_{\lambda_{x_{0}}}(x_0)} K_1^{\lambda_{x_{0}}}(x,y)\left[P_{x_0,\lambda_{x_{0}}}(y){\left(\frac{\lambda_{x_{0}}}{|y-x_0|}\right)}^{\tau}
u_{x_0,\lambda_{x_{0}}}^{q}(y)-P(y)u^{q}(y)\right]dy \nonumber\\
&>\int_{B_{\lambda_{x_{0}}}(x_0)} K_1^{\lambda_{x_{0}}}(x,y)\left[P_{x_0,\lambda_{x_{0}}}(y)u_{x_0,\lambda_{x_{0}}}^{q}(y)-P(y)u^{q}(y)\right]dy \nonumber\\
&=\int_{B_{\lambda_{x_{0}}}(x_0)}K_1^{\lambda_{x_{0}}}(x,y)u^{q}(y) \nonumber \\
&\quad \times\left(\int_{B_{\lambda_{x_{0}}}(x_0)}K_2^{\lambda_{x_{0}}}(y,z)
\left[\left(\frac{\lambda_{x_{0}}}{|z-x_0|}\right)^{\mu}u_{x_0,\lambda_{x_{0}}}^{p}(z)-u^{p}(z)\right]dz\right)dy \nonumber\\
&>\int_{B_{\lambda_{x_{0}}}(x_0)}K_1^{\lambda_{x_{0}}}(x,y)\left(\int_{B_{\lambda_{x_{0}}}(x_0)}K_2^{\lambda_{x_{0}}}(y,z)
\left[u_{x_0,\lambda_{x_{0}}}^{p}(z)-u^{p}(z)\right]dz\right)u^{q}(y)dy=0, \nonumber
\end{align}
which is a contradiction! Therefore, by Lemma \ref{lem1}, we must have, for every  $x_0\in\mathbb{R}^n$, the critical scale $\lambda_{x_0}=+\infty$. Then, by Lemma \ref{lem2} and the integrability $\int_{\mathbb{R}^n}\frac{u^{p}(x)}{|x|^{\sigma}}dx<+\infty$, we conclude that $u\equiv 0$. This is absurd since $u>0$ in $\mathbb{R}^{n}$. Therefore, there is no non-trivial nonnegative solution to IE \eqref{IE} in the subcritical cases.

Now we consider the critical case $p=\frac{2n-\sigma}{n-2s}$ and $q=\frac{n+2s-\sigma}{n-2s}$. Suppose that for every $x_0\in\mathbb{R}^{n}$, the critical scale $\lambda_{x_0}=+\infty$. Then, by Lemma \ref{lem2} and the integrability $\int_{\mathbb{R}^n}\frac{u^{p}(x)}{|x|^{\sigma}}dy<+\infty$, we conclude that $u\equiv 0$, which contradicts with $u>0$ in $\mathbb{R}^{n}$. Therefore, we must have, for every $x_0\in\mathbb{R}^n$, the critical scale $\lambda_{x_0}<+\infty$. Then by Proposition \ref{prop} and Lemma \ref{lem2}, we have
\begin{equation}\label{4-44}
  u(x)=C{\left(\frac{\mu}{1+\mu^{2}|x-\bar{x}|^2}\right)}^{\frac{n-2s}{2}}
\end{equation}
for some $C>0$, $\mu>0$ and $\bar{x}\in\mathbb{R}^n$. Consequently, we have proved, in the critical case, that any nontrivial nonnegative solution $u$ to IE \eqref{IE} must have the form \eqref{4-44} for some $C>0$, $\mu>0$ and $\bar{x}\in\mathbb{R}^n$. Furthermore, from (37) in Lemma 4.1 in \cite{DFHQW}, we have the following formula:
\begin{equation}\label{formula}
  \int_{\mathbb{R}^{n}}\frac{1}{|x-y|^{2\gamma}}\left(\frac{1}{1+|y|^{2}}\right)^{n-\gamma}dy=I(\gamma)\left(\frac{1}{1+|x|^{2}}\right)^{\gamma}
\end{equation}
for any $0<\gamma<\frac{n}{2}$, where $I(\gamma):=\frac{\pi^{\frac{n}{2}}\Gamma\left(\frac{n-2\gamma}{2}\right)}{\Gamma(n-\gamma)}$. By \eqref{formula}, \eqref{IE} and direct calculations, we deduce that the constant $C$ in \eqref{4-44} is given by
\begin{equation}\label{4-45}
  C=\left(\frac{1}{R_{2s,n}I\left(\frac{\sigma}{2}\right)I\left(\frac{n-2s}{2}\right)}\right)^{\frac{n-2s}{2(n+2s-\sigma)}}.
\end{equation}
This concludes our proof of Theorem \ref{Thm}.

\section{Proof of Theorem \ref{Thm3}}
In this section, using Theorem \ref{Thm0} and the arguments from Chen, Dai and Qin \cite{CDQ}, we will prove the Liouville properties in Theorem \ref{Thm3} in both critical order cases $s:=m+\frac{\alpha}{2}=\frac{n}{2}$ and super-critical order cases $s:=m+\frac{\alpha}{2}>\frac{n}{2}$.

We will prove Theorem \ref{Thm3} by using contradiction arguments. Suppose on the contrary that $u\geq0$ satisfies equation \eqref{PDE-sc} but $u$ is not identically zero, then there exists a point $\bar{x}\in\mathbb{R}^{n}$ such that $u(\bar{x})>0$. By Theorem \ref{Thm0}, we can deduce from $(-\Delta)^{\frac{\alpha}{2}}u\geq0$, $u\geq0$, $u(\bar{x})>0$ and maximum principle that
\begin{equation}\label{2-50-5}
  u(x)>0, \,\,\,\,\,\,\, \forall \,\, x\in\mathbb{R}^{n}.
\end{equation}
Moreover, by maximum principle and induction, we can also infer further from $(-\Delta)^{i+\frac{\alpha}{2}} u\geq0$ ($i=0,\cdots,m-1$), $u>0$ and equation \eqref{PDE-sc} that
\begin{equation}\label{2-51-5}
  (-\Delta)^{i+\frac{\alpha}{2}}u(x)>0, \,\,\,\,\,\,\,\, \forall \,\, i=0,\cdots,m-1, \,\,\,\, \forall \,\, x\in\mathbb{R}^{n}.
\end{equation}

Since $m+\frac{\alpha}{2}\geq\frac{n}{2}$, it follows immediately that either $m=\frac{n-1}{2}$ with $n\geq3$ odd, $m=\frac{n-2}{2}$ with $n\geq4$ even, or $m\geq\lceil\frac{n}{2}\rceil$, where $\lceil x\rceil$ denotes the least integer not less than $x$.

In the following, we will try to obtain contradictions by discussing the two different cases $m=\frac{n-1}{2}$ with $n\geq3$ odd or $m=\frac{n-2}{2}$ with $n\geq4$ even, and $m\geq\lceil\frac{n}{2}\rceil$ separately.

\emph{Case i): $m=\frac{n-1}{2}$ with $n\geq3$ odd or $m=\frac{n-2}{2}$ with $n\geq4$ even.} Since $m+\frac{\alpha}{2}\geq\frac{n}{2}$, we have $1\leq\alpha\leq2$ in the cases $m=\frac{n-1}{2}$ with $n\geq3$ odd and $\alpha=2$ in the cases $m=\frac{n-2}{2}$ with $n\geq4$ even. Now we will first show that $(-\Delta)^{m-1+\frac{\alpha}{2}}u$ satisfies the following integral equation
\begin{equation}\label{2c1-5}
  (-\Delta)^{m-1+\frac{\alpha}{2}}u(x)=\int_{\mathbb{R}^{n}}\frac{R_{2,n}}{|x-y|^{n-2}}\left(\frac{1}{|\cdot|^{\sigma}}\ast|u|^{p}\right)(y)u^{q}(y)dy, \,\,\,\,\,\,\,\,\,\, \forall \,\, x\in\mathbb{R}^{n},
\end{equation}
where the Riesz potential's constants $R_{\alpha,n}:=\frac{\Gamma(\frac{n-\alpha}{2})}{\pi^{\frac{n}{2}}2^{\alpha}\Gamma(\frac{\alpha}{2})}$ for $0<\alpha<n$.

To this end, for arbitrary $R>0$, let $f_{1}(u)(x):=\left(\frac{1}{|\cdot|^{\sigma}}\ast|u|^{p}\right)(x)u^{q}(x)$ and
\begin{equation}\label{2c2-5}
v_{1}^{R}(x):=\int_{B_R(0)}G^{2}_R(x,y)f_{1}(u)(y)dy,
\end{equation}
where the Green's function for $-\Delta$ on $B_R(0)$ is given by
\begin{equation}\label{Green-5}
  G^{2}_R(x,y)=R_{2,n}\bigg[\frac{1}{|x-y|^{n-2}}-\frac{1}{\big(|x|\cdot\big|\frac{Rx}{|x|^{2}}-\frac{y}{R}\big|\big)^{n-2}}\bigg], \,\,\,\, \text{if} \,\, x,y\in B_{R}(0),
\end{equation}
and $G^{2}_{R}(x,y)=0$ if $x$ or $y\in\mathbb{R}^{n}\setminus B_{R}(0)$. Then, we can derive that $v_{1}^{R}\in C^{2}(\mathbb{R}^{n})$ and satisfies
\begin{equation}\label{2c3-5}\\\begin{cases}
-\Delta v_{1}^{R}(x)=\left(\frac{1}{|\cdot|^{\sigma}}\ast|u|^{p}\right)(x)u^{q}(x),\ \ \ \ x\in B_R(0),\\
v_{1}^{R}(x)=0,\ \ \ \ \ \ \ x\in \mathbb{R}^{n}\setminus B_R(0).
\end{cases}\end{equation}
Let $w_{1}^R(x):=(-\Delta)^{m-1+\frac{\alpha}{2}}u(x)-v_{1}^R(x)$. By Theorem \ref{Thm0}, \eqref{PDE-sc} and \eqref{2c3-5}, we have $w_{1}^R\in C^{2}(\mathbb{R}^{n})$ and satisfies
\begin{equation}\label{2c4-5}\\\begin{cases}
-\Delta w_{1}^R(x)=0,\ \ \ \ x\in B_R(0),\\
w_{1}^{R}(x)>0, \,\,\,\,\, x\in \mathbb{R}^{n}\setminus B_R(0).
\end{cases}\end{equation}
By maximum principle, we deduce that for any $R>0$,
\begin{equation}\label{2c5-5}
  w_{1}^R(x)=(-\Delta)^{m-1+\frac{\alpha}{2}}u(x)-v_{1}^{R}(x)>0, \,\,\,\,\,\,\, \forall \,\, x\in\mathbb{R}^{n}.
\end{equation}
Now, for each fixed $x\in\mathbb{R}^{n}$, letting $R\rightarrow\infty$ in \eqref{2c5-5}, we have
\begin{equation}\label{2c6-5}
(-\Delta)^{m-1+\frac{\alpha}{2}}u(x)\geq\int_{\mathbb{R}^{n}}\frac{R_{2,n}}{|x-y|^{n-2}}f_{1}(u)(y)dy=:v_{1}(x)>0.
\end{equation}
Take $x=0$ in \eqref{2c6-5}, we get
\begin{equation}\label{2c7-5}
  \int_{\mathbb{R}^{n}}\left(\frac{1}{|\cdot|^{\sigma}}\ast|u|^{p}\right)(y)\frac{u^{q}(y)}{|y|^{n-2}}dy<+\infty.
\end{equation}
One can easily observe that $v_{1}\in C^{2}(\mathbb{R}^{n})$ is a solution of
\begin{equation}\label{2c8-5}
-\Delta v_{1}(x)=\left(\frac{1}{|\cdot|^{\sigma}}\ast|u|^{p}\right)(x)u^{q}(x),  \,\,\,\,\,\,\, x\in \mathbb{R}^n.
\end{equation}
Define $w_{1}(x):=(-\Delta)^{m-1+\frac{\alpha}{2}}u(x)-v_{1}(x)$. Then, by \eqref{PDE-sc}, \eqref{2c6-5} and \eqref{2c8-5}, we have $w_{1}\in C^{2}(\mathbb{R}^{n})$ and satisfies
\begin{equation}\label{2c9-5}\\\begin{cases}
-\Delta w_{1}(x)=0, \,\,\,\,\,  x\in \mathbb{R}^n,\\
w_{1}(x)\geq0, \,\,\,\,\,\,  x\in \mathbb{R}^n.
\end{cases}\end{equation}
From Liouville theorem for harmonic functions, we can deduce that
\begin{equation}\label{2c10-5}
   w_{1}(x)=(-\Delta)^{m-1+\frac{\alpha}{2}}u(x)-v_{1}(x)\equiv C_{1}\geq0.
\end{equation}
Therefore, we have
\begin{eqnarray}\label{2c11-5}
  (-\Delta)^{m-1+\frac{\alpha}{2}}u(x)&=&\int_{\mathbb{R}^{n}}\frac{R_{2,n}}{|x-y|^{n-2}}\left(\frac{1}{|\cdot|^{\sigma}}\ast|u|^{p}\right)(y)u^{q}(y)dy+C_{1} \\
  &=:&f_{2}(u)(x)>C_{1}\geq0. \nonumber
\end{eqnarray}

Next, for arbitrary $R>0$, let
\begin{equation}\label{2c12-5}
v_{2}^R(x):=\int_{B_R(0)}G^{2}_R(x,y)f_{2}(u)(y)dy.
\end{equation}
Then, we can get
\begin{equation}\label{2c13-5}\\\begin{cases}
-\Delta v_2^R(x)=f_{2}(u)(x),\ \ x\in B_R(0),\\
v_2^R(x)=0,\ \ \ \ \ \ \ x\in \mathbb{R}^{n}\setminus B_R(0).
\end{cases}\end{equation}
Let $w_2^R(x):=(-\Delta)^{m-2+\frac{\alpha}{2}}u(x)-v_2^R(x)$. By Theorem \ref{Thm0}, \eqref{2c11-5} and \eqref{2c13-5}, we have
\begin{equation}\label{2c14-5}\\\begin{cases}
-\Delta w_2^R(x)=0,\ \ \ \ x\in B_R(0),\\
w_2^R(x)>0, \,\,\,\,\, x\in \mathbb{R}^{n}\setminus B_R(0).
\end{cases}\end{equation}
By maximum principle, we deduce that for any $R>0$,
\begin{equation}\label{2c15-5}
  w_2^R(x)=(-\Delta)^{m-2+\frac{\alpha}{2}}u(x)-v_2^{R}(x)>0, \,\,\,\,\,\,\, \forall \,\, x\in\mathbb{R}^{n}.
\end{equation}
Now, for each fixed $x\in\mathbb{R}^{n}$, letting $R\rightarrow\infty$ in \eqref{2c15-5}, we have
\begin{equation}\label{2c16-5}
(-\Delta)^{m-2+\frac{\alpha}{2}}u(x)\geq\int_{\mathbb{R}^{n}}\frac{R_{2,n}}{|x-y|^{n-2}}f_{2}(u)(y)dy=:v_{2}(x)>0.
\end{equation}
Take $x=0$ in \eqref{2c16-5}, we get
\begin{equation}\label{2c17-5}
  \int_{\mathbb{R}^{n}}\frac{C_{1}}{|y|^{n-2}}dy\leq\int_{\mathbb{R}^{n}}\frac{f_{2}(u)(y)}{|y|^{n-2}}dy<+\infty,
\end{equation}
it follows easily that $C_{1}=0$, and hence we have proved \eqref{2c1-5}, that is,
\begin{equation}\label{2c18-5}
  (-\Delta)^{m-1+\frac{\alpha}{2}}u(x)=f_{2}(u)(x)=\int_{\mathbb{R}^{n}}\frac{R_{2,n}}{|x-y|^{n-2}}\left(\frac{1}{|\cdot|^{\sigma}}\ast|u|^{p}\right)(y)u^{q}(y)dy.
\end{equation}

One can easily observe that $v_{2}$ is a solution of
\begin{equation}\label{2c19-5}
-\Delta v_{2}(x)=f_{2}(u)(x),  \,\,\,\,\, x\in \mathbb{R}^n.
\end{equation}
Define $w_{2}(x):=(-\Delta)^{m-2+\frac{\alpha}{2}}u(x)-v_{2}(x)$, then it satisfies
\begin{equation}\label{2c20-5}\\\begin{cases}
-\Delta w_{2}(x)=0, \,\,\,\,\,  x\in \mathbb{R}^n,\\
w_{2}(x)\geq0, \,\,\,\,\,\,  x\in\mathbb{R}^n.
\end{cases}\end{equation}
From Liouville theorem for harmonic functions, we can deduce that
\begin{equation}\label{2c21-5}
   w_{2}(x)=(-\Delta)^{m-2+\frac{\alpha}{2}}u(x)-v_{2}(x)\equiv C_{2}\geq0.
\end{equation}
Therefore, we have proved that
\begin{equation}\label{2c22-5}
  (-\Delta)^{m-2+\frac{\alpha}{2}}u(x)=\int_{\mathbb{R}^{n}}\frac{R_{2,n}}{|x-y|^{n-2}}f_{2}(u)(y)dy+C_{2}=:f_{3}(u)(x)>C_{2}\geq0.
\end{equation}
By the same methods as above, we can prove that $C_{2}=0$, and hence
\begin{equation}\label{2c23-5}
  (-\Delta)^{m-2+\frac{\alpha}{2}}u(x)=f_{3}(u)(x)=\int_{\mathbb{R}^{n}}\frac{R_{2,n}}{|x-y|^{n-2}}f_{2}(u)(y)dy.
\end{equation}
Repeating the above argument, defining
\begin{equation}\label{2c24-5}
  f_{k+1}(u)(x):=\int_{\mathbb{R}^{n}}\frac{R_{2,n}}{|x-y|^{n-2}}f_{k}(u)(y)dy
\end{equation}
for $k=1,2,\cdots,m$, then by Theorem \ref{Thm0} and induction, we have
\begin{equation}\label{2c25-5}
  (-\Delta)^{m-k+\frac{\alpha}{2}}u(x)=f_{k+1}(u)(x)=\int_{\mathbb{R}^{n}}\frac{R_{2,n}}{|x-y|^{n-2}}f_{k}(u)(y)dy
\end{equation}
for $k=1,2,\cdots,m-1$, and
\begin{equation}\label{2c50-5}
  (-\Delta)^{\frac{\alpha}{2}}u(x)=\int_{\mathbb{R}^{n}}\frac{R_{2,n}}{|x-y|^{n-2}}f_{m}(u)(y)dy+C_{m}=f_{m+1}(u)(x)+C_{m}>C_{m}\geq0.
\end{equation}
For arbitrary $R>0$, let
\begin{equation}\label{2c12+5}
v_{m+1}^R(x):=\int_{B_R(0)}G^{\alpha}_R(x,y)\left(f_{m+1}(u)(y)+C_{m}\right)dy,
\end{equation}
where, in the cases $0<\alpha<2$, the Green's function for $(-\Delta)^{\frac{\alpha}{2}}$ on $B_R(0)$ is given by
\begin{equation}\label{2-8c+5}
G^\alpha_R(x,y):=\frac{C_{n,\alpha}}{|x-y|^{n-\alpha}}\int_{0}^{\frac{t_{R}}{s_{R}}}\frac{b^{\frac{\alpha}{2}-1}}{(1+b)^{\frac{n}{2}}}db
\,\,\,\,\,\,\,\,\, \text{if} \,\, x,y\in B_{R}(0)
\end{equation}
with $s_{R}=\frac{|x-y|^{2}}{R^{2}}$, $t_{R}=\left(1-\frac{|x|^{2}}{R^{2}}\right)\left(1-\frac{|y|^{2}}{R^{2}}\right)$, and $G^{\alpha}_{R}(x,y)=0$ if $x$ or $y\in\mathbb{R}^{n}\setminus B_{R}(0)$ (see \cite{K}). Then, we can get
\begin{equation}\label{2c13+5}\\\begin{cases}
(-\Delta)^{\frac{\alpha}{2}}v_{m+1}^R(x)=f_{m+1}(u)(x)+C_{m},\ \ x\in B_R(0),\\
v_{m+1}^R(x)=0,\ \ \ \ \ \ \ x\in \mathbb{R}^{n}\setminus B_R(0).
\end{cases}\end{equation}
Let $w_{m+1}^R(x):=u(x)-v_{m+1}^R(x)$. By Theorem \ref{Thm0}, \eqref{2c50-5} and \eqref{2c13+5}, we have
\begin{equation}\label{2c14+5}\\\begin{cases}
(-\Delta)^{\frac{\alpha}{2}}w_{m+1}^R(x)=0,\ \ \ \ x\in B_R(0),\\
w_{m+1}^R(x)>0, \,\,\,\,\, x\in \mathbb{R}^{n}\setminus B_R(0).
\end{cases}\end{equation}

By maximal principle for $-\Delta$ if $\alpha=2$ and Lemma \ref{max} if $0<\alpha<2$, we can deduce immediately from \eqref{2c14+5} that for any $R>0$,
\begin{equation}\label{2c15+5}
  w_{m+1}^R(x)=u(x)-v_{m+1}^{R}(x)>0, \,\,\,\,\,\,\, \forall \,\, x\in\mathbb{R}^{n}.
\end{equation}
Now, for each fixed $x\in\mathbb{R}^{n}$, letting $R\rightarrow\infty$ in \eqref{2c15+5}, we have
\begin{equation}\label{2c16+5}
u(x)\geq\int_{\mathbb{R}^{n}}\frac{R_{\alpha,n}}{|x-y|^{n-\alpha}}\left(f_{m+1}(u)(y)+C_{m}\right)dy>0.
\end{equation}
Take $x=0$ in \eqref{2c16+5}, we get
\begin{equation}\label{2c17+5}
  \int_{\mathbb{R}^{n}}\frac{C_{m}}{|y|^{n-\alpha}}dy\leq\int_{\mathbb{R}^{n}}\frac{f_{m+1}(u)(y)+C_{m}}{|y|^{n-\alpha}}dy<+\infty,
\end{equation}
it follows easily that $C_{m}=0$, and hence we have
\begin{equation}\label{2c18+5}
  (-\Delta)^{\frac{\alpha}{2}}u(x)=f_{m+1}(u)(x)=\int_{\mathbb{R}^{n}}\frac{R_{2,n}}{|x-y|^{n-2}}f_{m}(u)(y)dy,
\end{equation}
and
\begin{equation}\label{2c19+5}
  u(x)\geq\int_{\mathbb{R}^{n}}\frac{R_{\alpha,n}}{|x-y|^{n-\alpha}}f_{m+1}(u)(y)dy.
\end{equation}
In particular, it follows from \eqref{2c25-5}, \eqref{2c18+5} and \eqref{2c19+5} that
\begin{eqnarray}\label{2c51-5}
  && +\infty>(-\Delta)^{m-k+\frac{\alpha}{2}}u(0)=\int_{\mathbb{R}^{n}}\frac{R_{2,n}}{|y|^{n-2}}f_{k}(u)(y)dy \\
 \nonumber &\geq& \int_{\mathbb{R}^{n}}\frac{R_{2,n}}{|y^{k}|^{n-2}}\int_{\mathbb{R}^{n}}\frac{R_{2,n}}{|y^{k}-y^{k-1}|^{n-2}}\cdots
  \int_{\mathbb{R}^{n}}\frac{R_{2,n}}{|y^{2}-y^{1}|^{n-2}}\left(\frac{1}{|\cdot|^{\sigma}}\ast|u|^{p}\right)(y^{1})u^{q}(y^{1})dy^{1}\cdots dy^{k}
\end{eqnarray}
for $k=1,2,\cdots,m$, and
\begin{eqnarray}\label{formula-5}
  && +\infty>u(0)\geq\int_{\mathbb{R}^{n}}\frac{R_{\alpha,n}}{|y|^{n-\alpha}}f_{m+1}(u)(y)dy\geq\int_{\mathbb{R}^{n}}\frac{R_{\alpha,n}}{|y^{m+1}|^{n-\alpha}}\times \\
 \nonumber &&\left(\int_{\mathbb{R}^{n}}\frac{R_{2,n}}{|y^{m+1}-y^{m}|^{n-2}}\cdots
  \int_{\mathbb{R}^{n}}\frac{R_{2,n}}{|y^{2}-y^{1}|^{n-2}}\left(\frac{1}{|\cdot|^{\sigma}}\ast|u|^{p}\right)(y^{1})u^{q}(y^{1})dy^{1}\cdots dy^{m}\right)dy^{m+1}.
\end{eqnarray}
From the properties of Riesz potential, for any $\alpha_{1},\alpha_{2}\in(0,n)$ such that $\alpha_{1}+\alpha_{2}\in(0,n)$, one has(see \cite{Stein})
\begin{equation}\label{2c26-5}
  \int_{\mathbb{R}^{n}}\frac{R_{\alpha_{1},n}}{|x-y|^{n-\alpha_{1}}}\cdot\frac{R_{\alpha_{2},n}}{|y-z|^{n-\alpha_{2}}}dy
=\frac{R_{\alpha_{1}+\alpha_{2},n}}{|x-z|^{n-(\alpha_{1}+\alpha_{2})}}.
\end{equation}
By applying \eqref{2c26-5} and direct calculations, we obtain that
\begin{eqnarray}\label{2c27-5}
  && \int_{\mathbb{R}^{n}}\frac{R_{2,n}}{|y^{m+1}-y^{m}|^{n-2}}\cdots
\int_{\mathbb{R}^{n}}\frac{R_{2,n}}{|y^{3}-y^{2}|^{n-2}}\cdot\frac{R_{2,n}}{|y^{2}-y^{1}|^{n-2}}dy^{2}\cdots dy^{m} \\
 \nonumber &=& \frac{R_{2m,n}}{|y^{m+1}-y^{1}|^{n-2m}}.
\end{eqnarray}

Now, we can deduce from \eqref{formula-5}, \eqref{2c27-5} and Fubini's theorem that
\begin{eqnarray}\label{contradiction-5}
  &&+\infty>u(0)\geq\int_{\mathbb{R}^{n}}\frac{R_{\alpha,n}}{|y^{m+1}|^{n-\alpha}}\left(\int_{\mathbb{R}^{n}}\frac{R_{n-\alpha,n}}{|y^{m+1}-y^{1}|^{n-2m}}
  \left(\frac{1}{|\cdot|^{\sigma}}\ast|u|^{p}\right)(y^{1})u^{q}(y^{1})dy^{1}\right) \\
 \nonumber &&\qquad\quad dy^{m+1}=\frac{1}{(2\pi)^{n}}\int_{\mathbb{R}^{n}}\frac{1}{|y|^{n-\alpha}}\left(\int_{\mathbb{R}^{n}}\frac{1}{|y-z|^{n-2m}}
 \left(\frac{1}{|\cdot|^{\sigma}}\ast|u|^{p}\right)(z)u^{q}(z)dz\right)dy.
\end{eqnarray}

We will get a contradiction from \eqref{contradiction-5}. Indeed, if we assume that $u$ is not identically zero, then by \eqref{2-50-5}, $u>0$ in $\mathbb{R}^{n}$. Hence by the integrability \eqref{2c7-5}, we have
\begin{equation}\label{2c60-5}
0<C_{0}:=\int_{\mathbb{R}^{n}}\left(\frac{1}{|\cdot|^{\sigma}}\ast|u|^{p}\right)(z)\frac{u^{q}(z)}{|z|^{n-2}}dz<+\infty.
\end{equation}
For any given $|y|\geq 3$, if $|z|\geq\big(\ln|y|\big)^{-\frac{1}{n-2}}$, then one has immediately
\begin{equation}\label{2c61-5}
|y-z|\leq|y|+|z|\leq\left(|y|\big(\ln|y|\big)^{\frac{1}{n-2}}+1\right)|z|\leq 2|y|\big(\ln|y|\big)^{\frac{1}{n-2}}|z|.
\end{equation}
Thus it follows from \eqref{2c60-5} and \eqref{2c61-5} that, there exists a $R_{0}\geq3$ sufficiently large such that, for any $|y|\geq R_{0}$, we have
\begin{eqnarray}\label{2c63-5}
&&\int_{\mathbb{R}^{n}}\left(\frac{1}{|\cdot|^{\sigma}}\ast|u|^{p}\right)(z)\frac{u^{q}(z)}{|y-z|^{n-2m}}dz \\
\nonumber &\geq& \frac{1}{2^{n-2m}|y|^{n-2m}\ln|y|}\int_{|z|\geq\left(\ln|y|\right)^{-\frac{1}{n-2}}}\left(\frac{1}{|\cdot|^{\sigma}}\ast|u|^{p}\right)(z)\frac{u^{q}(z)}{|z|^{n-2}}dz \\
\nonumber &\geq& \frac{1}{2^{n-2m+1}|y|^{n-2m}\ln|y|}\int_{\mathbb{R}^{n}}\left(\frac{1}{|\cdot|^{\sigma}}\ast|u|^{p}\right)(z)\frac{u^{q}(z)}{|z|^{n-2}}dz
\geq\frac{C_{0}}{2^{n-2m+1}|y|^{n-2m}\ln|y|}.
\end{eqnarray}

As a consequence, we can finally deduce from \eqref{contradiction-5}, \eqref{2c63-5}, $1\leq\alpha\leq2$ if $m=\frac{n-1}{2}$ with $n\geq3$ odd and $\alpha=2$ if $m=\frac{n-2}{2}$ with $n\geq4$ even that
\begin{equation}\label{final-5}
 +\infty>u(0)\geq\frac{C_{0}}{2^{n-2m+1}(2\pi)^{n}}\int_{|y|\geq R_{0}}\frac{1}{|y|^{2n-\alpha-2m}\ln|y|}dy=+\infty,
\end{equation}
which is a contradiction. Therefore $u\equiv0$ in $\mathbb{R}^{n}$. This proves Theorem \ref{Thm3} in Case i): $m=\frac{n-1}{2}$ with $n\geq3$ odd or $m=\frac{n-2}{2}$ with $n\geq4$ even.

\emph{Case ii): $m\geq\lceil\frac{n}{2}\rceil$.} Let
\begin{equation}\label{2c24-05}
  f_{k+1}(u)(x):=\int_{\mathbb{R}^{n}}\frac{R_{2,n}}{|x-y|^{n-2}}f_{k}(u)(y)dy
\end{equation}
for $k=1,2,\cdots,\lceil\frac{n}{2}\rceil$, by a quite similar way as in the proof for Case i), we can infer from Theorem \ref{Thm0} and induction that
\begin{equation}\label{2c25-05}
  (-\Delta)^{m-k+\frac{\alpha}{2}}u(x)=f_{k+1}(u)(x)=\int_{\mathbb{R}^{n}}\frac{R_{2,n}}{|x-y|^{n-2}}f_{k}(u)(y)dy
\end{equation}
for $k=1,2,\cdots,\lceil\frac{n}{2}\rceil-1$, and
\begin{equation}\label{2c50-05}
 (-\Delta)^{m-\lceil\frac{n}{2}\rceil+\frac{\alpha}{2}}u(x)\geq f_{\lceil\frac{n}{2}\rceil+1}(u)(x)=\int_{\mathbb{R}^{n}}\frac{R_{2,n}}{|x-y|^{n-2}}f_{\lceil\frac{n}{2}\rceil}(u)(y)dy.
\end{equation}
In particular, it follows from \eqref{2c25-05} and \eqref{2c50-05} that
\begin{eqnarray}\label{2c51-05}
  && +\infty>(-\Delta)^{m-k+\frac{\alpha}{2}}u(0)=\int_{\mathbb{R}^{n}}\frac{R_{2,n}}{|y|^{n-2}}f_{k}(u)(y)dy \\
 \nonumber &\geq& \int_{\mathbb{R}^{n}}\frac{R_{2,n}}{|y^{k}|^{n-2}}\int_{\mathbb{R}^{n}}\frac{R_{2,n}}{|y^{k}-y^{k-1}|^{n-2}}\cdots
  \int_{\mathbb{R}^{n}}\frac{R_{2,n}}{|y^{2}-y^{1}|^{n-2}}\left(\frac{1}{|\cdot|^{\sigma}}\ast|u|^{p}\right)(y^{1})u^{q}(y^{1})dy^{1}\cdots dy^{k}
\end{eqnarray}
for $k=1,2,\cdots,\lceil\frac{n}{2}\rceil-1$, and
\begin{eqnarray}\label{formula-05}
  && +\infty>(-\Delta)^{m-\lceil\frac{n}{2}\rceil+\frac{\alpha}{2}}u(0)\geq\int_{\mathbb{R}^{n}}\frac{R_{2,n}}{|y|^{n-2}}f_{\lceil\frac{n}{2}\rceil}(u)(y)dy \\
 \nonumber && \qquad \geq\int_{\mathbb{R}^{n}}\frac{R_{2,n}}{|y^{\lceil\frac{n}{2}\rceil}|^{n-2}}\Bigg(\int_{\mathbb{R}^{n}}\frac{R_{2,n}}{|y^{\lceil\frac{n}{2}\rceil}-y^{\lceil\frac{n}{2}\rceil-1}|^{n-2}}\cdots
  \int_{\mathbb{R}^{n}}\frac{R_{2,n}}{|y^{2}-y^{1}|^{n-2}} \\
  \nonumber &&\qquad\quad \times\left(\frac{1}{|\cdot|^{\sigma}}\ast|u|^{p}\right)(y^{1})u^{q}(y^{1})dy^{1}\cdots dy^{\lceil\frac{n}{2}\rceil-1}\Bigg)dy^{\lceil\frac{n}{2}\rceil}.
\end{eqnarray}
By applying the formula \eqref{2c26-5} and direct calculations, we obtain that
\begin{eqnarray}\label{2c27-05}
  && \int_{\mathbb{R}^{n}}\frac{R_{2,n}}{|y^{\lceil\frac{n}{2}\rceil}-y^{\lceil\frac{n}{2}\rceil-1}|^{n-2}}\cdots
\int_{\mathbb{R}^{n}}\frac{R_{2,n}}{|y^{3}-y^{2}|^{n-2}}\cdot\frac{R_{2,n}}{|y^{2}-y^{1}|^{n-2}}dy^{2}\cdots dy^{\lceil\frac{n}{2}\rceil-1} \\
 \nonumber &=& \frac{R_{2\lceil\frac{n}{2}\rceil-2,n}}{|y^{\lceil\frac{n}{2}\rceil}-y^{1}|^{n-2\lceil\frac{n}{2}\rceil+2}}.
\end{eqnarray}

Now, we can deduce from \eqref{formula-05}, \eqref{2c27-05} and Fubini's theorem that
\begin{eqnarray}\label{contradiction-05}
  && +\infty>(-\Delta)^{m-\lceil\frac{n}{2}\rceil+\frac{\alpha}{2}}u(0) \\
  \nonumber &\geq&\int_{\mathbb{R}^{n}}\frac{R_{2,n}}{|y^{\lceil\frac{n}{2}\rceil}|^{n-2}}
  \left(\int_{\mathbb{R}^{n}}\frac{R_{2\lceil\frac{n}{2}\rceil-2,n}}{|y^{\lceil\frac{n}{2}\rceil}-y^{1}|^{n-2\lceil\frac{n}{2}\rceil+2}}
  \left(\frac{1}{|\cdot|^{\sigma}}\ast|u|^{p}\right)(y^{1})u^{q}(y^{1})dy^{1}\right)dy^{\lceil\frac{n}{2}\rceil} \\
 \nonumber &=& C_{n}\int_{\mathbb{R}^{n}}\frac{1}{|y|^{n-2}}\left(\int_{\mathbb{R}^{n}}\frac{1}{|y-z|^{n-2\lceil\frac{n}{2}\rceil+2}}
 \left(\frac{1}{|\cdot|^{\sigma}}\ast|u|^{p}\right)(z)u^{q}(z)dz\right)dy.
\end{eqnarray}

We will get a contradiction from \eqref{contradiction-05}. To do this, let $\tau(n):=n-2\lceil\frac{n}{2}\rceil+2\in\{1,2\}$, then it follows from \eqref{2c60-5} and \eqref{2c61-5} that, there exists a $R_{0}\geq3$ sufficiently large such that, for any $|y|\geq R_{0}$,
\begin{eqnarray}\label{2c63-05}
&&\int_{\mathbb{R}^{n}}\frac{1}{|y-z|^{\tau(n)}}\left(\frac{1}{|\cdot|^{\sigma}}\ast|u|^{p}\right)(z)u^{q}(z)dz \\
\nonumber &\geq& \frac{1}{2^{\tau(n)}|y|^{\tau(n)}\ln|y|}\int_{|z|\geq\left(\ln|y|\right)^{-\frac{1}{n-2}}}\frac{1}{|z|^{n-2}}
\left(\frac{1}{|\cdot|^{\sigma}}\ast|u|^{p}\right)(z)u^{q}(z)dz \\
\nonumber &\geq& \frac{1}{2^{\tau(n)+1}|y|^{\tau(n)}\ln|y|}\int_{\mathbb{R}^{n}}\left(\frac{1}{|\cdot|^{\sigma}}\ast|u|^{p}\right)(z)
\frac{u^{q}(z)}{|z|^{n-2}}dz\geq\frac{C_{0}}{2^{\tau(n)+1}|y|^{\tau(n)}\ln|y|}.
\end{eqnarray}
Therefore, we can finally deduce from \eqref{contradiction-05} and \eqref{2c63-05} that
\begin{equation}\label{final-05}
 +\infty>(-\Delta)^{m-\lceil\frac{n}{2}\rceil+\frac{\alpha}{2}}u(0)\geq\frac{C_{0}C_{n}}{2^{\tau(n)+1}}\int_{|y|\geq R_{0}}\frac{1}{|y|^{n-2+\tau(n)}\ln|y|}dy=+\infty,
\end{equation}
which is a contradiction again. Therefore, $u\equiv0$ in $\mathbb{R}^{n}$ in Case ii): $m\geq \lceil\frac{n}{2}\rceil$. This concludes our proof of Theorem \ref{Thm3}.


\begin{thebibliography}{99}

\bibitem{Be} J. Bertoin, {\it L\'{e}vy Processes}, Cambridge Tracts in Mathematics, \textbf{121}, Cambridge University Press, Cambridge, 1996.

\bibitem{BGM} E. Berchio, F. Gazzola and E. Mitidieri, {\it Positivity preserving property for a class of biharmonic elliptic problems}, J. Diff. Equations, \textbf{229} (2006), 1-23.

\bibitem{BKN} K. Bogdan, T. Kulczycki and A. Nowak, {\it Gradient estimates for harmonic and $q$-harmonic functions of symmetric stable processes}, Illinois J. Math., \textbf{46} (2002), 541-556.

\bibitem{CD} D. Cao and W. Dai, {\it Classification of nonnegative solutions to a bi-harmonic equation with Hartree type nonlinearity}, Proc. Royal Soc. Edinburgh-A: Math., \textbf{149} (2019), 979-994.

\bibitem{CDQ0} D. Cao, W. Dai and G. Qin, {\it Super poly-harmonic properties, Liouville theorems and classification of nonnegative solutions to equations involving higher-order fractional Laplacians}, preprint, submitted, arXiv: 1905.04300.

\bibitem{CDQ} W. Chen, W. Dai and G. Qin, {\it Liouville type theorems, a priori estimates and existence of solutions for critical order Hardy-H\'{e}non equations in $\mathbb{R}^n$}, preprint, submitted, arXiv: 1808.06609.

\bibitem{CDZ} D. Cao, W. Dai and Y. Zhang, {\it Existence and symmetry of positive solutions to 2-D Schr\"{o}dinger-Newton equations}, preprint, submitted, 2019.

\bibitem{CF} W. Chen and Y. Fang, {\it A Liouville type theorem for poly-harmonic Dirichlet problems in a half space}, Adv. Math., \textbf{229} (2012), 2835-2867.

\bibitem{CFL} W. Chen, Y. Fang and C. Li, {\it Super poly-harmonic property of solutions for Navier boundary problems on a half space}, J. Funct. Anal., \textbf{265} (2013), 1522-1555.

\bibitem{CFY} W. Chen, Y. Fang and R. Yang, {\it Liouville theorems involving the fractional Laplacian on a half space}, Adv. Math., \textbf{274} (2015), 167-198.

\bibitem{CGS} L. Caffarelli, B. Gidas and J. Spruck, {\it Asymptotic symmetry and local behavior of semilinear elliptic equations with critical Sobolev growth}, Comm. Pure Appl. Math., \textbf{42} (1989), 271-297.

\bibitem{CL1} W. Chen and C. Li, {\it On Nirenberg and related problems - a necessary and sufficient condition}, Comm. Pure Appl. Math., \textbf{48} (1995), 657-667.

\bibitem{CL0} W. Chen and C. Li, {\it Moving planes, moving spheres, and a priori estimates}, J. Differential Equations, \textbf{195} (2003), no. 1, 1-13.

\bibitem{CL2} W. Chen and C. Li, {\it Classification of positive solutions for nonlinear differential and integral systems with critical exponents}, Acta Math. Sci., \textbf{29B} (2009), 949-960.

\bibitem{CL4} W. Chen and C. Li, {\it Super poly-harmonic property of solutions for PDE systems and its applications}, Comm. Pure Appl. Anal., \textbf{12} (2013), 2497-2514.

\bibitem{CL3} D. Cao and H. Li, {\it High energy solutions of the Choquard equation}, Disc. Cont. Dyn. Syst. - A, \textbf{38} (2018), no. 6, 3023-3032.

\bibitem{CLL} W. Chen, C. Li and Y. Li, {\it A direct method of moving planes for the fractional Laplacian}, Adv. Math., \textbf{308} (2017), 404-437.

\bibitem{CLO} W. Chen, C. Li and B. Ou, {\it Classification of solutions for an integral equation}, Comm. Pure Appl. Math., \textbf{59} (2006), 330-343.

\bibitem{CLM} W. Chen, Y. Li and P. Ma, {\it The Fractional Laplacian}, World Scientific Publishing Co. Pte. Ltd., 2019, 350pp, https://doi.org/10.1142/10550.

\bibitem{CLZ} W. Chen, Y. Li and R. Zhang, {\it A direct method of moving spheres on fractional order equations}, J. Funct. Anal., \textbf{272} (2017), no. 10, 4131-4157.

\bibitem{Co} P. Constantin, {\it Euler equations, Navier-Stokes equations and turbulence, in Mathematical Foundation of Turbulent Viscous Flows}, Vol. 1871 of Lecture Notes in Math., 1-43, Springer, Berlin, 2006.

\bibitem{CS} L. Caffarelli and L. Silvestre, {\it An extension problem related to the fractional Laplacian}, Comm. PDEs, \textbf{32} (2007), 1245-1260.

\bibitem{CT} X. Cabr\'{e} and J. Tan, {\it Positive solutions of nonlinear problems involving the square root of the Laplacian}, Adv. Math., \textbf{224} (2010), 2052-2093.

\bibitem{CV} L. Caffarelli and L. Vasseur, {\it Drift diffusion equations with fractional diffusion and the quasi-geostrophic equation}, \textbf{171} (2010), no. 3, 1903-1930.

\bibitem{CW} S. Cingolani and T. Weth, {\it On the planar Schr\"{o}dinger-Poisson system}, Ann. Inst. H. Poincar\'{e} Anal. Non Lin\'{e}aire, \textbf{33} (2016), no. 1, 169-197.

\bibitem{CY} S.-Y. A. Chang and P. C. Yang, {\it On uniqueness of solutions of $n$-th order differential equations in conformal geometry}, Math. Res. Lett., \textbf{4} (1997), 91-102.

\bibitem{DFHQW} W. Dai, Y. Fang, J. Huang, Y. Qin and B. Wang, {\it Regularity and classification of solutions to static Hartree equations involving fractional Laplacians}, Discrete and Continuous Dynamical Systems - A, \textbf{39} (2019), no. 3, 1389-1403.

\bibitem{DFQ} W. Dai, Y. Fang and G. Qin, {\it Classification of positive solutions to fractional order Hartree equations via a direct method of moving planes}, J. Diff. Equations, \textbf{265} (2018), 2044-2063.

\bibitem{DL} W. Dai and Z. Liu, {\it Classification of nonnegative solutions to static Schr\"{o}dinger-Hartree and Schr\"{o}dinger-Maxwell equations with combined nonlinearities}, Calc. Var. \& PDEs, \textbf{58} (2019), no. 4: 156, https://doi.org/10.1007/s00526-019-1595-z.

\bibitem{DQ} W. Dai and G. Qin, {\it Classification of nonnegative classical solutions to third-order equations}, Adv. Math., \textbf{328} (2018), 822-857.

\bibitem{DQ3} W. Dai and G. Qin, {\it Liouville type theorem for critical order H\'{e}non-Lane-Emden type equations on a half space and its applications}, preprint, arXiv: 1811.00881.

\bibitem{FL} J. Frohlich, E. Lenzmann, {\it Mean-field limit of quantum bose gases and nonlinear Hartree equation}, in: Sminaire E. D. P. (2003-2004), Expos nXVIII. 26p.

\bibitem{GNN1} B. Gidas, W. Ni and L. Nirenberg, {\it Symmetry and related properties via maximum principle}, Comm. Math. Phys., \textbf{68} (1979), 209-243.

\bibitem{JLX} Q. Jin, Y. Y. Li and H. Xu, {\it Symmetry and Asymmetry: The Method of Moving Spheres}, Adv. Differential Equations, \textbf{13} (2007), no. 7, 601-640.

\bibitem{Karpman} V. L. Karpman, {\it Stabilization of soliton instabilities by high-order dispersion: fourth order nonlinear Schr\"{o}dinger-type equations}, Phys. Rev. E 53, \textbf{2} (1996), 1336-1339.

\bibitem{K} T. Kulczycki, {\it Properties of Green function of symmetric stable processes}, Probability and Mathematical Statistics, \textbf{17} (1997), 339-364.

\bibitem{Lei} Y. Lei, {\it Qualitative analysis for the Hartree-type equations}, SIAM J. Math. Anal., \textbf{45} (2013), 388-406.

\bibitem{Li} Y. Y. Li, {\it Remark on some conformally invariant integral equations: the method of moving spheres}, J. European Math. Soc., \textbf{6} (2004), 153-180.

\bibitem{Lieb} E. H. Lieb, {\it Sharp constants in the Hardy-Littlewood-Sobolev and related inequalities}, Ann. of Math., \textbf{118} (1983), no. 2, 349-374.

\bibitem{Lin} C. S. Lin, {\it A classification of solutions of a conformally invariant fourth order equation in $\mathbb{R}^{n}$}, Comment. Math. Helv., \textbf{73} (1998), 206-231.

\bibitem{L} P. L. Lions, {\it The concentration-compactness principle in the calculus of variations. The locally compact case, parts1 and 2}, Ann. Inst. H. Poincar\'{e} Anal. Non Lin\'{e}aire., \textbf{1} (1984), no. 2, 109-145, no. 4, 223-283.

\bibitem{L1} P. L. Lions, {\it The concentration-compactness principle in the calculus of variations. The limit case, parts1 and 2}, Revista Math. Iberoamericana, \textbf{1} (1985), no. 1, 145-201, no. 2, 45-121.

\bibitem{Liu} S. Liu, {\it Regularity, symmetry, and uniqueness of some integral type quasilinear equations}, Nonlinear Anal., \textbf{71} (2009), 1796-1806.

\bibitem{LMZ} D. Li, C. Miao and X. Zhang, {\it The focusing energy-critical Hartree equation}, J. Diff. Equations, \textbf{246} (2009), 1139-1163.

\bibitem{LS} E. Lieb and B. Simon, {\it The Hartree-Fock theory for Coulomb systems}, Comm. Math. Phys., \textbf{53} (1977), 185-194.

\bibitem{LZ} Y. Li and M. Zhu, {\it Uniqueness theorems through the method of moving spheres}, Duke Math. J., \textbf{80} (1995), 383-417.

\bibitem{LZ1} Y. Li and L. Zhang, {\it Liouville type theorems and Harnack type inequalities for semilinear elliptic equations}, J. Anal. Math, \textbf{90} (2003), 27-87.

\bibitem{MS} V. Moroz and J. Van Schaftingen, {\it Groundstates of nonlinear Choquard equations: existence, qualitative properties and decay asymptotics}, J. Funct. Anal., \textbf{265} (2013), no. 2, 153-184.

\bibitem{MS1} V. Moroz and J. Van Schaftingen, {\it Existence of groundstates for a class of nonlinear Choquard equations}, Trans. Amer. Math. Soc., \textbf{367} (2015), no. 9, 6557-6579.

\bibitem{MXZ} C. Miao, G. Xu and L. Zhao, {\it Global wellposedness and scattering for the focusing energy-critical nonlinear Schr\"{o}dinger equations of fourth order in the radial case}, J. Diff. Equations, \textbf{246} (2009), 3715-3749.

\bibitem{MXZ3} C. Miao, G. Xu, and L. Zhao, {\it Global well-posedness, scattering and blow-up for the energy-critical, focusing Hartree equation in the radial case}, Colloq. Math., \textbf{114} (2009), 213-236.

\bibitem{MZ} L. Ma and L. Zhao, {\it Classification of positive solitary solutions of the nonlinear Choquard equation}, Arch. Rational Mech. Anal., \textbf{195} (2010), no. 2, 455-467.

\bibitem{Pa} P. Padilla, {\it On some nonlinear elliptic equations}, Thesis, Courant Institute, 1994.

\bibitem{Stein} E. M. Stein, {\it Singular integrals and differentiability properties of functions}, Princeton Landmarks in Mathematics, Princeton University Press, Princeton, New Jersey, 1970.

\bibitem{Serrin} J. Serrin, {\it A symmetry problem in potential theory}, Arch. Rational Mech. Anal., \textbf{43} (1971), 304-318.

\bibitem{S} L. Silvestre, {\it Regularity of the obstacle problem for a fractional power of the Laplace operator}, Comm. Pure Appl. Math., \textbf{60} (2007), 67-112.

\bibitem{WX} J. Wei and X. Xu, {\it Classification of solutions of higher order conformally invariant equations}, Math. Ann., \textbf{313} (1999), no. 2, 207-228.

\bibitem{XL} D. Xu and Y. Lei, {\it Classification of positive solutions for a static Schr\"{o}dinger-Maxwell equation with fractional Laplacian}, Applied Math. Letters, \textbf{43} (2015), 85-89.

\bibitem{Xu} X. Xu, {\it Exact solutions of nonlinear conformally invariant integral equations in $\mathbb{R}^{3}$}, Adv. Math., \textbf{194} (2005), 485-503.

\bibitem{ZCCY} R. Zhuo, W. Chen, X. Cui and Z. Yuan, {\it A Liouville theorem for the fractional Laplacian}, arXiv: 1401.7402.

\end{thebibliography}
\end{document}